\documentclass[12pt]{amsart}
\usepackage{amssymb}
\usepackage{amsmath}

\renewcommand{\subjclass}[1]
{\thanks{\emph{2000 Mathematics Subject Classification:}~#1}}

\newtheorem{theorem}{Theorem}[section]
\newtheorem{lemma}{Lemma}[section]
\newtheorem{prp}{Proposition}[section]
\newcommand{\K}{\mathbb{K}}
\newcommand{\C}{\mathbb{C}}
\newcommand{\Q}{\mathbb{Q}}
\newcommand{\Z}{\mathbb{Z}}
\newcommand{\R}{\mathbb{R}}
\newcommand{\w}{\mathbf{w}}
\newcommand{\x}{\mathbf{x}}
\newcommand{\ee}{\mathbf{e}}

\begin{document}

\title{Optimal systems of fundamental $S$-units for LLL-reduction}


\author{L. Hajdu}
\thanks{Research supported in part by the Hungarian Academy
of Sciences and by the OTKA grants T48791 and K67580.}
\address{L. Hajdu \newline
         \indent Institute of Mathematics, University of Debrecen
         \newline
         \indent Number Theory Research Group, Hungarian Academy
         of Sciences
         \newline
         \indent and University of Debrecen \newline
         \indent H-4010 Debrecen, P.O. Box 12, Hungary}
\email{hajdul\char'100math.klte.hu}

\begin{abstract}
We show that a particular parameter plays a vital role in the
resolution of $S$-unit equations, at the stage where LLL-reduction
is applied. We define the notion of optimal system of fundamental
$S$-units (with respect to this parameter), and prove that such a
system exists and can be effectively constructed. Applying our
results and methods, one can obtain much better bounds for the
solutions of $S$-unit equations after the reduction step, than
earlier. We briefly also discuss some effects of our results on
the method of Wildanger and Smart for the resolution of
$S$-unit equations.
\end{abstract}

\maketitle

\section{Introduction}

$S$-unit equations play a central role in the theory of
Diophantine equations. On the one hand, there are many Diophantine
problems which naturally give rise to $S$-unit equations. On the
other hand, several types of classical Diophantine equations (such
as e.g. norm form equations, discriminant form equations, index
form equations) can be reduced to such equations. Here we only
refer to the papers, survey papers and books \cite{Gy1},
\cite{ShTi}, \cite{EGyST}, \cite{EGy}, \cite{Gy2}, \cite{Sp},
\cite{G} and the references given there.

It is well-known (see e.g. \cite{ShTi}, \cite{EGyST}, \cite{Gy2},
\cite{ESS} and the references there) that under general conditions
the number of solutions of $S$-unit equations is finite. However,
if the number of variables is more than two, then no bound is
known for the solutions themselves. Moreover, a conjecture of
Cerlienco, Mignotte and Piras \cite{CMP} states that such
equations are algorithmically unsolvable. In case of $S$-unit
equations in two variables, by the help of Baker's method the
solutions can also be bounded (see e.g. \cite{Gy0}, \cite{ShTi},
\cite{EGyST}, \cite{Gy2}, \cite{Gy3}, \cite{GyY} and the
references given there). Further, based on LLL-type results of
de Weger (cf. \cite{deW}) and an enumeration method of Wildanger
\cite{W}, there exists an efficient algorithm for determining all
solutions of any particular equations in this case; see
\cite{Sm2}.

The practical solution of $S$-unit equations in two unknowns (and
also of some other types, e.g. of index form equations) consists
of three main steps (see e.g. \cite{P1}, \cite{GPP1}, \cite{GPP2},
\cite{W}, \cite{Sm2}, \cite{GGy}, \cite{BGGy}, and the references
given there). First, applying Baker's method an initial upper
bound $C_{ini}$ is obtained for the unknowns (which are in the
exponents). For the best known bounds see \cite{M} in the complex
and \cite{Y} in the $p$-adic case, respectively, and also the
references given there. Roughly speaking, at this principal stage
the "large" solutions are excluded. Though from the theoretical
point of view this first step is the deepest and most important
one, for practical purposes (i.e. for listing explicitly all
solutions) the Baker-type bound is (typically) too large. Hence
before going for the solutions, some further reduction is needed.
In the second step applying some variants of the LLL-algorithm,
one can reduce the bound $C_{ini}$ considerably, to get a much
smaller bound $C_{red}$ for the solutions (see e.g. \cite{Sm2} and
the references given there). This stage is frequently referred to
as getting rid of the "medium" size solutions. However, if the
number of variables (the rank of the $S$-unit group) is "large",
even this reduced bound can be too high for getting all solutions,
at least if one would like to simply apply some primitive listing
algorithm. So to find the solutions effectively, some clever
enumeration is needed. Such an algorithm has been worked out by
Wildanger \cite{W} in the complex case, and extended and adopted
by Smart \cite{Sm2} for the $S$-unit case. The application of
these algorithms as a third step (ideally) leads to the explicit
solution of the original $S$-unit equation.

In this paper we focus on the second stage of this procedure. It
turns out (in fact it is widely known already) that in pushing
down the reduced bound $C_{red}$ as much as possible, a certain
parameter plays an important role. The aim of the present paper is
to show that there is an optimal choice for this crucial
parameter, and to give a method to actually find it. Hence
ultimately we are able to get much better bounds $C_{red}$ in case
of particular equations than earlier. Having established our
results, we briefly discuss their possible effects on the method
of Wildanger \cite{W} and Smart \cite{Sm2} for the practical
resolution of $S$-unit equations. We note that the results
presented in the paper lead to certain improvements of the method.

To formulate our results clearly, we need some preparation. For this
purpose we introduce some (standard) notation.

\subsection{Valuations}
\label{i1}

Let $\K$ be an algebraic number field of degree $n$, and let
$M_\K$ be the set of places on $\K$. Choose from every place $v\in
M_\K$ a valuation $|.|_v$ in the following way. If $v$ is infinite
and corresponds to an embedding $\sigma:\K\to\C$ then for every
$\alpha\in\K$ let
$$
|\alpha|_v=
\begin{cases}
|\sigma(\alpha)|, \mbox{ if } \sigma \mbox{ is real},\\
|\sigma(\alpha)|^2, \mbox{ if } \sigma \mbox{ is complex}.
\end{cases}
$$
Further, if $v$ is finite and corresponds to a prime ideal $P$ of
$\K$, then for every $\alpha\in\K$ put
$$
|\alpha|_v=
\begin{cases}
0, \mbox{ if } \alpha=0,\\
\mbox{Norm}(P)^{-\mbox{\scriptsize ord}_P(\alpha)}, \mbox{
otherwise}.
\end{cases}
$$
By these choices we have the product formula, that is for every
$\alpha\in\K$, $\alpha\neq 0$
\begin{equation}
\label{pf}
\prod\limits_{v\in M_\K} |\alpha|_v = 1
\end{equation}
holds. Let $S=\{v_1,\dots,v_s\}$ be a finite subset of $M_\K$,
containing all the infinite places, and write
$$
U_S=\{\varepsilon : |\varepsilon|_v=1 \mbox{ for all } v\in
M_\K\setminus S\}
$$
for the set of $S$-units. As is well-known, $U_S$ is a finitely
generated group of rank $s-1$, containing the unit group of the
ring of integers of $\K$.

\subsection{$S$-unit equations}
\label{i2}

Let $\alpha_1,\alpha_2$ be fixed non-zero elements of $\K$.
Consider the so-called $S$-unit equation
\begin{equation}
\label{sueq}
\alpha_1 x_1 + \alpha_2 x_2 = 1
\end{equation}
in two unknowns $x_1,x_2\in U_S$. It is well-known (see e.g.
\cite{EGyST}) that if $x_1,x_2$ is a solution to equation
\eqref{sueq} then we have
\begin{equation}
\label{xiexpand}
x_i = \varepsilon_0^{b_{i,0}}\varepsilon_1^{b_{i,1}}\dots
\varepsilon_{s-1}^{b_{i,s-1}} \mbox{   for } i=1,2.
\end{equation}
Here $\varepsilon_1,\dots,\varepsilon_{s-1}$ is a fundamental
system of $S$-units, and $\varepsilon_0$ is a root of unity in
$\K$. Put $B_i=\max\limits_{1\leq j\leq s-1} |b_{i,j}|$ for
$i=1,2$ and $B=\max(B_1,B_2)$. Clearly, without loss of generality
we may assume that $B=B_1$. From \eqref{xiexpand}, for all $v\in
M_\K$
$$
\log |x_1|_v=\sum\limits_{j=1}^{s-1} b_{1,j}\log |\varepsilon_j|_v
$$
holds. In particular, we have
\begin{equation}
\label{matrixeq}
\begin{pmatrix}
\log |\varepsilon_1|_{v_1} & \dots & \log
|\varepsilon_{s-1}|_{v_1}\\
\vdots & \ddots & \vdots \\
\log |\varepsilon_1|_{v_s} & \dots & \log
|\varepsilon_{s-1}|_{v_s}
\end{pmatrix}
\cdot
\begin{pmatrix}
b_{1,1}\\
\vdots\\
b_{1,s-1}\\
\end{pmatrix}
=
\begin{pmatrix}
\log |x_1|_{v_1}\\
\vdots\\
\log |x_1|_{v_s}\\
\end{pmatrix}
.
\end{equation}
From the above equality \eqref{matrixeq} one easily gets that
\begin{equation}
\label{mainineq}
B_1 \leq C^* \max\limits_{1\leq j\leq s} |\log |x_1|_{v_j}|
\end{equation}
with some constant $C^*$, depending only on our fundamental system
of $S$-units. (The constant $C^*$ is the crucial object from the
point of view of the present paper; we shall come back to this
point a bit later.) Then, by the help of the product formula
\eqref{pf}, using standard arguments (see e.g. \cite{EGyST},
\cite{Gy3}) we immediately get that
\begin{equation}
\label{x1vfelso}
|x_1|_{v_t} \leq \exp\left(\frac{-B_1}{(s-1)C^*}\right)
\end{equation}
holds for some $t\in\{1,\dots,s\}$. Hence from \eqref{sueq} we get
\begin{equation}
\label{bfelso}
|1-\alpha_2 x_2|_{v_t} \leq |\alpha_1|_{v_t}
\exp\left(\frac{-B_1}{(s-1)C^*}\right).
\end{equation}
Now using a Baker-type result (e.g. Matveev \cite{M} if $v_t$ is
infinite and Yu \cite{Y} if $v_t$ is finite, respectively) we get
something like
\begin{equation}
\label{balso}
|1-\alpha_2 x_2|_{v_t} > \exp(-C_0\log(B_2)).
\end{equation}
Here $C_0$ is some constant depending only on
$\alpha_1,\alpha_2,\K,S$. The inequalities \eqref{bfelso} and
\eqref{balso} by our assumption $B=B_1$ yield an initial upper
bound $C_{ini}$ for $B$.

\subsection{The importance of $C^*$}
\label{i3}

To get the reduced bound $C_{red}$ by the LLL-algorithm, one
starts from inequality \eqref{bfelso} (or a variant of it)
together with the already known information $B<C_{ini}$.
Interestingly, the finally obtained reduced bound $C_{red}$
depends very heavily on the constant $C^*$; the dependence is
closely linear. This phenomenon should hopefully be
(heuristically) clear from the examples presented in this paper,
but already is well-known for experts for a long time; see e.g.
the remarks of Tzanakis and de Weger \cite{TW} pp. 239-240, and
Smart \cite{Sm1} p. 823. Hence it seems to be worth to try to keep
$C^*$ as small as possible. Apparently, so far this point remained
more or less untouched, and the calculation of $C^*$ is usually
done in a rather casual way. Namely, in all occurrences in the
literature the standard choice is to take something like
\begin{equation}
\label{oldc*}
C^*:= \min\limits_{1\leq j\leq s} ||R_j^{-1}||
\end{equation}
(with certain refinements at some places). Here $R_j$ is the
(invertible) matrix obtained by deleting the $j$-th row of the
matrix at the left hand side of \eqref{matrixeq}, and $||A||$
stands for the row norm of a $k_1\times k_2$ type real matrix
$A=(a_{ij})_{\underset{1\leq j\leq k_2}{1\leq i\leq k_1}}$, i.e.
$||A||=\max\limits_{1\leq i\leq k_1}\sum\limits_{j=1}^{k_2}
|a_{ij}|$. Using e.g. Cramer's rule, one can easily see that this
choice of $C^*$ is appropriate to have \eqref{mainineq}.

However, as it turns out, this choice of $C^*$ can be rather far
from being optimal, which results in a much worse value for
$C_{red}$ than possible. In the second section we show a very
simple way to get an instant improvement upon the above choice.
Further, we show that in fact the best $C^*$ value exists, and
depends only on the choice of the fundamental system of $S$-units.
Finally, we prove that one can explicitly determine an optimal
fundamental system of $S$-units (yielding the best value for
$C^*$), and we give a (relatively) efficient algorithm for finding
such a system. As the steps of our arguments and methods are
connected in a rather organic way, we do not start with listing
the main theorems, we prefer to formulate our results in a
"linear" way. However, in order not to break the presentation, we
give the proofs in a later section. In Section \ref{opsys} we give
an algorithm which finds an optimal system of fundamental
$S$-units for LLL-reduction. In the fourth section of the paper we
provide some numerical examples, including the bounds $C_{red}$
obtained by the old method and by the new one. In the fifth
section we give the proofs of our results. The sixth section is
devoted to a brief discussion about the effects of our results on
the method of Wildanger \cite{W} and Smart \cite{Sm2} for the
practical solution of $S$-unit equations. Finally, in the
Appendix on the one hand we outline the reduction methods, and on
the other hand we indicate how one can adjust the method developed
in the paper if the valuations are not chosen in the "standard"
way.

Note that our method can be adopted to the case where in
\eqref{sueq} not a full system of fundamental $S$-units are
involved, or the $S$-units form only an independent system.

Finally, we mention that a similar type investigation has been
performed about Mordell-Weil bases of elliptic curves by Stroeker
and Tzanakis \cite{StTz}, to reduce the final bound for the
integral solutions of elliptic equations.

\section{Optimizing $C^*$}

We keep our notation from the previous section. Further, let
$F=(\varepsilon_1,\dots,\varepsilon_{s-1})$ be a system of
fundamental $S$-units of $\K$, and define the $s\times (s-1)$
matrix $R_F$ by
$$
R_F:=
\begin{pmatrix}
\log |\varepsilon_1|_{v_1} & \dots & \log
|\varepsilon_{s-1}|_{v_1}\\
\vdots & \ddots & \vdots \\
\log |\varepsilon_1|_{v_s} & \dots & \log
|\varepsilon_{s-1}|_{v_s}
\end{pmatrix}
.
$$
Let $R_F'$ be any $(s-1)\times s$ matrix such that
\begin{equation}
\label{rf'}
R_F'\cdot R_F=E_{(s-1)\times (s-1)},
\end{equation}
with the identity matrix of size $(s-1)\times (s-1)$ on the right
hand side. The importance of the matrices $R_F'$ becomes clear in
view of the following simple observation: starting from
\eqref{matrixeq}, by \eqref{rf'} we get
$$
\begin{pmatrix}
b_{1,1}\\
\vdots\\
b_{1,s-1}\\
\end{pmatrix}
=
R_F'
\cdot
\begin{pmatrix}
\log |x_1|_{v_1}\\
\vdots\\
\log |x_1|_{v_s}\\
\end{pmatrix}
.
$$
Hence in \eqref{mainineq} we can take $C^*$ to be $||R_F'||$, with
any matrix $R_F'$. So we define the norm $N(F)$ of the system $F$
by
$$
N(F):=\min\limits_{R_F'} ||R_F'||,
$$
where $R_F'$ runs through the matrices for which \eqref{rf'} is
valid. (We shall see later that the minimum does exist.) The
system $F$ is called optimal for LLL-reduction, if $N(F)$ is
minimal among all choices of fundamental systems of $S$-units. As
we shall also see, the minimum of $N(F)$ also exists. Hence we put
\begin{equation}
\label{newc*}
C^*:= \min\limits_F N(F),
\end{equation}
where $F$ runs through all systems of fundamental $S$-units. Then
we have \eqref{mainineq}, of course using the optimal system $F$
in \eqref{xiexpand}. In order to compare this choice of $C^*$ with
the earlier one, we write $N_{old}(F)$ for the choice of $C^*$ in
\eqref{oldc*}. Note that \eqref{oldc*} just means that we
unnecessarily restrict ourselves to matrices $R_F'$ having a
constant zero column.

In the following proposition we describe the structure of the
matrices $R_F'$, for fixed $F$.

\begin{prp}
\label{p1} Let $R_F'$ be a matrix for which \eqref{rf'} is valid.
Then for each $i\in\{1,\dots,s-1\}$ there exists a $u_i\in\R$
such that the $i$-th row of $R_F'$ is of the form $\w_i-u_i\cdot
{\mathbf 1}$ with $\w_i=(w_{i,1},\dots,w_{i,s-1},0)$, where
${\mathbf 1}$ is the $s$-tuple with all entries equal to $1$, and
$$
\begin{pmatrix}
w_{1,1}&\dots&w_{1,s-1}\\
\vdots&\ddots&\vdots\\
w_{s-1,1}&\dots&w_{s-1,s-1}
\end{pmatrix}
=
\begin{pmatrix}
\log |\varepsilon_1|_{v_1} & \dots & \log
|\varepsilon_{s-1}|_{v_1}\\
\vdots & \ddots & \vdots \\
\log |\varepsilon_1|_{v_{s-1}} & \dots & \log
|\varepsilon_{s-1}|_{v_{s-1}}
\end{pmatrix}
^{-1}.
$$
\end{prp}

The next result provides a simple tool to calculate $N(F)$ for a
fixed system $F$ of fundamental $S$-units. Before its formulation,
we need a new notation. Define the central norm $\vert\x\vert_C$
of $\x\in\R^n$, $\x=(x_1,\dots,x_n)$ in the following way. Let
$y_1,\dots,y_n$ be a rearrangement of $x_1,\dots,x_n$ such that
$y_1\leq\dots\leq y_n$. Then put
\begin{equation}
\label{cnorm}
\vert\x\vert_C=\sum\limits_{i=1}^n |y_l-y_i|
\end{equation}
where $l=\lfloor(n+1)/2\rfloor$. The number $y_l$ is called the
center of $\x$.

\begin{prp}
\label{p2}
Using the previous notation, we have
$$
N(F)=\max\limits_{1\leq i\leq s-1} |\w_i|_C.
$$
\end{prp}

Note that the above result already gives a tool to improve upon
the earlier choice of $C^*$ in \eqref{oldc*}. Indeed, if one does
not interested in finding the optimal system of fundamental
$S$-units for LLL-reduction, but insists on his favorite system
$F$, taking $C^*$ to be $N(F)$ rather than $N_{old}(F)$ in
\eqref{mainineq}, is already an improvement. By Proposition
\ref{p2} the calculation of $N(F)$ takes only a fraction of time.

The next result (together with its proof) shows that there exists
an optimal system of fundamental $S$-units indeed, and such a
system can be effectively determined.

\begin{theorem}
\label{t1} For any positive real constant $c$ there are only
finitely many systems $F$ of fundamental $S$-units such that
$N(F)\leq c$. Further, all such systems can be effectively
determined.
\end{theorem}

\section{An algorithm to determine an optimal system $F$}
\label{opsys}

In this section we outline an algorithm to determine an optimal
system of fundamental $S$-units for LLL-reduction. The algorithm
consists of two parts. First, by a heuristic method we (hopefully)
obtain the best system $F$, then we check that our choice is best
possible indeed.

In the first step we start from an arbitrary system $F_0$ of
fundamental $S$-units. (Such a system can be found e.g. by Magma
\cite{BCP}, but one can also use KASH \cite{KASH} or PARI/GP
\cite{PARI}). Then we calculate the values of $w_{i,j}$ and choose
$u_i$ to be the center of $\w_i$ $(1\leq i,j\leq s-1)$ in
Proposition \ref{p1}. Hence for the corresponding matrix
$R_{F_0}'$ we have that $||R_{F_0}'||$ is minimal (see the proof
of Proposition \ref{p2}). To find an optimal system $F$, in fact
we need to find an unimodular $(s-1)\times (s-1)$ type matrix
$A_0$ such that $N(F_0 A_0)$ is minimal.

For any row vector ${\mathbf b}=(b_1,\dots,b_n)$ write ${\mathbf
b}^*=(b_1-u,\dots,b_n-u)$, where $u$ is the center of ${\mathbf
b}$. We call ${\mathbf b}^*$ the centralization of ${\mathbf b}$.
If $B$ is a matrix then let $B^*$ denote the matrix obtained from
$B$ by replacing all its rows by their centralizations. Observe
that for any matrices $A,B$ (of appropriate sizes) we have
$(A\cdot B)^*=(A\cdot B^*)^*$. Hence our problem reduces to
finding the following minimum:
\begin{equation}
\label{mina} \min\limits_A ||(A R_{F_0}')^*||
\end{equation}
where $A$ runs through all the unimodular matrices of type
$(s-1)\times (s-1)$. If this minimum is taken at some matrix $A'$,
we simply have $A_0=(A')^{-1}$.

We determine the above minimum with a heuristic algorithm, which
seems to work well. The algorithm produces a sequence of systems
$F_0,F_1,F_2,\dots$ with $N(F_0)>N(F_1)>N(F_2)>\dots$, and
terminates at some point. Suppose that we have already made $i$
steps, and currently we are working with $F_i$. Let
$$
R_{F_i}'=
\begin{pmatrix}
w_{1,1}^{(i)}&\dots&w_{1,s}^{(i)}\\
\vdots&\ddots&\vdots\\
w_{s-1,1}^{(i)}&\dots&w_{s-1,s}^{(i)}
\end{pmatrix}
$$
belong to the optimal choice, i.e. $||R_{F_i}'||=N(F_i)$. Write
${\mathbf w_t^{(i)}}=(w_{t,1}^{(i)},\dots,w_{t,s}^{(i)})$ for the
$t$-th row of $R_{F_i}'$ $(t=1,\dots,s-1)$. Suppose that
$||R_{F_i}'||=\sum\limits_{l=1}^s |w_{j,l}^{(i)}|$ (that is the
$j$-th row of $R_{F_i}'$ yields the row norm $||R_{F_i}'||$). To
improve upon $||R_{F_i}'||$ we need to find an unimodular matrix
$A$ such that for some row ${\mathbf a}=(a_1,\dots,a_{s-1})$ of
$A$ we have $a_j\neq 0$, and further
$$
\left|\sum\limits_{t=1}^{s-1} a_t{\mathbf w_t^{(i)}}\right|_C <
||R_{F_i}'||.
$$
Indeed, otherwise we cannot "replace" the vector ${\mathbf
w_j^{(i)}}$ by any "shorter" one, and hence we cannot improve upon
$||R_{F_i}'||$. In the first part of the algorithm we try to find
such a row vector ${\mathbf {a}}$ of a simple shape, and then we
iterate the procedure. More precisely, we consider the minimum
\eqref{mina}, however, only for $A$ running through the unimodular
matrices which are different from the identity matrix only in
their $j$-th row; namely, the $(j,j)$-th entry of $A$ equals $1$,
and all the other entries in its $j$-th row may assume the values
$-1,0,1$ only. Having the minimum at $A'$, we define $F_{i+1}$ by
$F_{i+1}=F_i(A')^{-1}$, and then repeat the procedure. By Theorem
\ref{t1} we know that this algorithm terminates, and produces a
system $F_k$ as output.

Now we should check that the final system is optimal. (Note that
this was the case for every example we tested the algorithm for,
so we think that this is the typical phenomenon.) For this we need
some preparation; in fact we need to establish two simple
properties of the central norm.

\begin{lemma}
\label{l1}
The central norm is homogeneous, that is for any $\x\in\R^n$
and $t\in\R$ we have
$\vert t\x\vert_C=|t|\vert\x\vert_C$.
\end{lemma}

\begin{lemma}
\label{l2} Let $e_i\in\R^n$ denote the vector with all coordinates
$0$, except for the $i$-th one which is $1$ $(i=1,\dots,n-1)$, and
put
$$
T:=\{\ee_i:i=1,\dots,n-1\} \cup \{-\ee_i:i=1,\dots,n-1\} \cup
\{\ee_0,-\ee_0\}
$$
where $\ee_0=\ee_1+\dots+\ee_{n-1}$. Further, set
$$
H:=\{\x\in\R^n:x_n=0\ {\rm and}\ \vert\x\vert_C\leq 1\}
$$
where $x_n$ is the last entry of $\x$. Then $H$ is the convex hull
of $T$.
\end{lemma}

As a corollary of the above two lemmas we get the following
statement, which will be needed in the last step of our algorithm.

\begin{lemma}
\label{l3} Let $F$ be a system of fundamental $S$-units, and
choose an $R_F'$ as before. Then for any (integral) unimodular
matrix $A$ of type $(s-1)\times (s-1)$ with $||(A R_F')^*||<N(F)$
we have that the row vectors of $A$ belong to the convex hull of
the set
$$
\{\pm N(F){\mathbf b}_1,\dots, \pm N(F){\mathbf b}_s\},
$$
where ${\mathbf b}_i$ is the $i$-th row of $R_F$ $(i=1,\dots,s)$.
\end{lemma}

Now we can outline the final part of our algorithm. Having the
output $F_k$ by the first stage of the procedure, we calculate
$N(F_k)$ (by the help of Proposition \ref{p2}). Then using Lemma
\ref{l3}, we get an upper bound $c_0$ for each entry of the
possible unimodular matrices $A$. Typically, this bound is very
small (around at most 2-3). We keep the notation introduced above
Lemma \ref{l1}. In particular, we assume again that the value of
$||R_{F_k}'||$ belongs to the $j$-th row of $R_{F_k}'$. Then we
check all row vectors ${\mathbf a}=(a_1,\dots,a_{s-1})$ such that
$a_j>0$ and all the entries of ${\mathbf a}$ are at most $c_0$ in
absolute value. If for all such vectors ${\mathbf a}$ we have
$\left|{\mathbf a}R_{F_k}'\right|_C\geq N(F_k)$, then the value of
$N(F_k)$ cannot be improved and $F_k$ is an optimal system.
Otherwise, for every appropriate row vectors ${\mathbf a}$ we
check all unimodular matrices with entries at most $c_0$ in
absolute value, containing ${\mathbf a}$ as a row. In case we can
improve upon $N(F_k)$, we switch back to the first part of the
algorithm (with the improved system), then return to this second
part later on, and so on. By Theorem \ref{t1} we eventually get an
optimal system of fundamental $S$-units. We mention that the
second part is much more time consuming than the first one, at
least if $s$ is "large". (It is not surprising in the light of the
observation that the number of row vectors ${\mathbf a}$ to be
checked is $c_0(2c_0+1)^{s-2}$ at this stage.) However, note that
in all cases we encountered the second part has never been
necessary in the sense that the system $F_k$ obtained by the first
part of the algorithm has already been optimal.

We conclude this section by two remarks. First we note that our
simple heuristic algorithm works so well is very probably due to
the fact that the systems of fundamental units provided e.g. by
Magma are already LLL-reduced, hence the initial system is
"closely" optimal. The other thing we mention is that our
algorithm is not optimized, it probably can be improved. Further,
maybe there are other (possibly more efficient) ways to get the
optimal system. For example, after having an original system $F_0$
one can search for the minimizing unimodular matrix $A$ in
\eqref{mina} in the following way. Let ${\mathcal L}$ denote the
lattice spanned by the vectors ${\mathbf w_t^{(0)}}$
$(t=1,\dots,s-1)$. If $F_0$ is not optimal then there is some
unimodular matrix $A$ such that for all row vectors ${\mathbf a}$
of $A$ we have $\sum\limits_{t=1}^s |a_t'|<N(F_0)$ where ${\mathbf
w_a}:=(a_1',\dots,a_s')=({\mathbf a}R_{F_0}')^*$. This implies
that ${\mathbf w_a}$ is a vector of the lattice ${\mathcal L}$ of
Euclidean length less than $\sqrt{s}N(F_0)$. Hence all the
appropriate row vectors ${\mathbf a}$ can be efficiently
determined by the algorithm of Fincke and Pohst \cite{FP}. Having
the set of possible vectors ${\mathbf a}$, we can build up the
appropriate unimodular matrices $A$, and we can find the minimum
\eqref{mina}. However, as we mentioned, the original algorithm was
efficient enough for our present purposes, so here we do not take
up the problem of optimizing the method.

\section{Examples}

In this section we present some examples to illustrate our method
and also to give some numerical evidence why
the parameter $C^*$ is so important. In fact we have worked out
several examples, but we give here only four, out of which three
can be found in the literature. We exhibit a "proper" $S$-unit
equation (i.e. with a finite valuation involved) and three pure
unit equations. Among the latter ones the first two correspond to
totally real fields, while the last one to a totally complex
field. These examples are of larger size. We mention that in case
of smaller examples, our algorithm worked with the same
efficiency, but much faster. However, naturally in the "small"
cases the gain using the new values of $C^*$ is certainly smaller.

In each example we consider an $S$-unit equation of the shape
\begin{equation}
\label{sueq2}
x_1 + x_2 = 1
\end{equation}
in $x_1,x_2\in U_S$, for some particular choice of $\K$ and $S$.
Just as in \eqref{xiexpand}, we write
\begin{equation}
\label{xiexpand2}
x_i = \varepsilon_0^{b_{i,0}}\varepsilon_1^{b_{i,1}}\dots
\varepsilon_{s-1}^{b_{i,s-1}} \mbox{   for } i=1,2.
\end{equation}
Note that here the system $\varepsilon_1,\dots,\varepsilon_{s-1}$
is certainly not fixed during the examples, in fact we use two
different systems in each example (corresponding to different
values of $C^*$).

In every example, starting from an initial system $F_0$ of
fundamental $S$-units, using our algorithm we determine an optimal
system $F_k$. (Note that at the examples from the literature we
use the system $F_0$ from the corresponding papers, otherwise we
used Magma \cite{BCP} to get an $F_0$.) Remark that the first
stage of our algorithm typically takes a few seconds (up to a few
minutes when $s$ is larger) to terminate and provide an optimal
system $F_k$. However, to check that $F_k$ is optimal indeed by
the second part of the algorithm takes much more time (around $30$
minutes for "large" $s$). An optimized and more sophisticated
version of the algorithm would probably work better, but we do not
take up this question here.

\subsection*{Description of the tables}

We provide a table for each example, containing several data. We
give the values of the reduced bounds $C_{red}$ corresponding to
the choices $C^*=N_{old}(F_0)$, $N(F_0)$ and $N(F_k)$,
respectively. Naturally, in the first two cases we use the system
$F_0$ in \eqref{xiexpand2}, while in case of $C^*=N(F_k)$ the
system $F_k$ is used. To execute the reduction steps, we used
Lemmas \ref{redpadic} and \ref{redabs} given in the Appendix. Note
that these reduction lemmas have very many variants in the
literature. We chose these ones because they are relatively simply
formulated, and they are appropriate for our present purposes. To
make the reduction, when there was no available initial upper
bound for $B$, as it is not an important point from the point of
view of the present paper, instead of going through Baker's method
we just started with the (plausible) bound $B<10000=:C_{ini}$.

In the tables we provide the ratios of the $C^*$ values, as well
(more precisely the ratios corresponding to the actual $C^*$ and
$N_{old}(F_0)$). We also indicate the ratios of the $C_{red}$
values (that is, the ratios corresponding to the actual $C_{red}$
and the reduced bound corresponding to the choice
$C^*=N_{old}(F_0)$). Note that these ratios are remarkably close
to each other, which shows that the dependence of $C_{red}$ is
very nearly linear in $C^*$. This phenomenon is not surprising in
view of Lemmas \ref{redpadic} and \ref{redabs}: the new lower
bound for $B$ in each iteration is almost linear in $1/C_1$ in
both cases - and $1/C_1$ is linear in $C^*$. (For the details cf.
subsection \ref{redu} of the Appendix.)

Finally, we introduce an indicator called "domain ratio" to
compare the remaining domains to be checked after the reduction.
This indicator is defined in the natural way, i.e. as
$$
\left(\frac{2C_{red}^{(1)}+1}{2C_{red}^{(0)}+1}\right)^{2s-2}.
$$
Here the constants $C_{red}$ are the reduced constants
($C_{red}^{(0)}$ corresponds to the choice $C^*=N_{old}(F_0)$ and
$C_{red}^{(1)}$ to the actual choice), and the exponent $2s-2$ by
\eqref{xiexpand2} is just the number of variables (as $b_{1,0}$
and $b_{2,0}$ do not really count). Note that in our examples this
indicator shows that the size of the domain to be checked using
the new method is a tiny fraction of the size of the domain
remaining by the old method.

\vskip.2cm

\noindent{\bf Example 1.} This example is from \cite{Sm2}. Let
$\K=\Q(\vartheta)$ where $\vartheta^8+1=0$. Let $S$ be the set
containing the infinite valuations of $\K$, and a finite valuation
corresponding to the prime ideal $P=(\pi)$ with $\pi=1-\vartheta$.
We have $N_{\K/\Q}(\pi)=2$, and a fundamental system of $S$-units
is given by
$$
\varepsilon_1=\vartheta^2+\vartheta^4+\vartheta^6,
\varepsilon_2=-\vartheta^2-\vartheta^3-\vartheta^4,
\varepsilon_3=1+\vartheta^3-\vartheta^5,
\varepsilon_4=\pi=1-\vartheta
$$
(see \cite{Sm2}). Note that the element
$\varepsilon_0=-\vartheta^7$ generates the sixteen roots of unity
of $\K$. Put
$$
F_0=(\varepsilon_1,\varepsilon_2,\varepsilon_3,\varepsilon_4).
$$
A simple calculation shows that $N_{old}(F_0)=1.442695\dots$.
Further, we also have $N(F_0)=1.442695\dots$. Then, executing the
above algorithm starting with $F_0$, in three steps we get a new
system of fundamental $S$-units $F_3$ given by the transformation
$$
\left(
\begin{smallmatrix}
1 & 0 & 0 & 0\\
0 & 1 & 0 & 0\\
0 & 0 & 1 & 0\\
0 & 1 & 1 & 1
\end{smallmatrix}
\right)
\cdot
\left(
\begin{smallmatrix}
1 & 0 & 0 & 0\\
0 & 1 & 0 & 0\\
0 & 1 & 1 & 0\\
0 & 1 & 1 & 1
\end{smallmatrix}
\right)
\cdot
\left(
\begin{smallmatrix}
1 & 0 & \hfill 0 & \hfill 0\\
1 & 1 & \hfill -1 & \hfill -1\\
0 & 0 & \hfill 1 & \hfill 0\\
0 & 0 & \hfill 0 & \hfill 1
\end{smallmatrix}
\right)
=
\left(
\begin{smallmatrix}
1 & 0 &\hfill 0 &\hfill 0\\
1 & 1 &\hfill -1 &\hfill -1\\
1 & 1 &\hfill 0 &\hfill -1\\
2 & 2 &\hfill -1 &\hfill -1
\end{smallmatrix}
\right)
.
$$
Actually, we have
$$
F_3=(\varepsilon_1\varepsilon_2\varepsilon_3\varepsilon_4^2,
\varepsilon_2\varepsilon_3\varepsilon_4^2,
\varepsilon_2^{-1}\varepsilon_4^{-1},
\varepsilon_2^{-1}\varepsilon_3^{-1}\varepsilon_4^{-1}).
$$
A simple calculation gives that $N(F_3)=0.931871\dots$. The
second part of our algorithm reveals that $F_3$ is an optimal
system of fundamental $S$-units.

Then, using Lemmas \ref{redpadic} and \ref{redabs} from the
Appendix, we get the reduced bounds in the table below. Note that
we need to perform the reduction for each choice of the
valuations, and that the worst case belongs to the finite
valuation in $S$. We also mention that working with $F_0$, we used
the initial upper bound $B\leq 1066$, from \cite{Sm2}. In case of
$F_3$, using the inverse of the basis transformation matrix, from
this we have $B\leq 3198$. We started our reduction with these
bounds. We summarize the results of our computations in Table 1.
For the definitions of the entries of the table (and also in case
of the other tables) see the preceding subsection.

\begin{table}[htb]\centering
\vskip.15cm
\begin{tabular}{|l|l|l|l|}
\cline{1-4} \vbox to1,88ex{\vspace{1pt}\vfil\hbox to12,80ex{\hfil
\hfil}} &\vbox to1,88ex{\vspace{1pt}\vfil\hbox to20,80ex{\hfil
using $F_0$ and $N_{old}(F_0)$ \hfil}} & \vbox
to1,88ex{\vspace{1pt}\vfil\hbox to18,80ex{\hfil using $F_0$ and
$N(F_0)$ \hfil}} & \vbox to1,88ex{\vspace{1pt}\vfil\hbox
to18,80ex{\hfil using $F_3$ and $N(F_3)$ \hfil}} \\

\cline{1-4} \vbox to1,88ex{\vspace{1pt}\vfil\hbox to12,80ex{\hfil
$C^*$\hfil}} &\vbox to1,88ex{\vspace{1pt}\vfil\hbox
to20,80ex{\hfil 1.442695...\hfil}} & \vbox
to1,88ex{\vspace{1pt}\vfil\hbox to18,80ex{\hfil 1.442695...\hfil}}
& \vbox to1,88ex{\vspace{1pt}\vfil\hbox
to18,80ex{\hfil 0.931871...\hfil}} \\

\cline{1-4} \vbox to1,88ex{\vspace{1pt}\vfil\hbox to12,80ex{\hfil
$C_{red}$\hfil}} &\vbox to1,88ex{\vspace{1pt}\vfil\hbox
to20,80ex{\hfil 1031\hfil}} & \vbox
to1,88ex{\vspace{1pt}\vfil\hbox to18,80ex{\hfil 1031\hfil}} &
\vbox to1,88ex{\vspace{1pt}\vfil\hbox to18,80ex{\hfil 651\hfil}}
\\

\cline{1-4} \vbox to1,88ex{\vspace{1pt}\vfil\hbox to12,80ex{\hfil
$C^*$ ratio\hfil}} &\vbox to1,88ex{\vspace{1pt}\vfil\hbox
to20,80ex{\hfil 1\hfil}} & \vbox to1,88ex{\vspace{1pt}\vfil\hbox
to18,80ex{\hfil 1\hfil}} & \vbox to1,88ex{\vspace{1pt}\vfil\hbox
to18,80ex{\hfil 0.645923...\hfil}} \\

\cline{1-4} \vbox to1,88ex{\vspace{1pt}\vfil\hbox to12,80ex{\hfil
$C_{red}$ ratio\hfil}} &\vbox to1,88ex{\vspace{1pt}\vfil\hbox
to20,80ex{\hfil 1\hfil}} & \vbox to1,88ex{\vspace{1pt}\vfil\hbox
to18,80ex{\hfil 1\hfil}} & \vbox to1,88ex{\vspace{1pt}\vfil\hbox
to18,80ex{\hfil 0.631425...\hfil}} \\

\cline{1-4} \vbox to1,88ex{\vspace{1pt}\vfil\hbox to12,80ex{\hfil
domain ratio\hfil}} &\vbox to1,88ex{\vspace{1pt}\vfil\hbox
to20,80ex{\hfil 1\hfil}} & \vbox to1,88ex{\vspace{1pt}\vfil\hbox
to18,80ex{\hfil 1\hfil}} & \vbox to1,88ex{\vspace{1pt}\vfil\hbox
to18,80ex{\hfil 0.025325...\hfil}} \\

\cline{1-4}
\end{tabular}
\vskip.2cm
\centerline{Table 1}
\end{table}

\vskip.2cm

\noindent{\bf Example 2.} The data of this example is from
\cite{GGy} (the authors considered a different equation). Let
$\K=\Q(\vartheta)$ where
$\vartheta^{10}-15\vartheta^8+\vartheta^7+66\vartheta^6+
\vartheta^5-96\vartheta^4-7\vartheta^3+37\vartheta^2+12\vartheta+1=0$.
Let $S$ be the set of infinite valuations of $\K$. An integral
basis of $\K$ is given by $\omega_1,\dots,\omega_{10}$ with
$\omega_i=\vartheta^{i-1}$ $(i=1,\dots,9)$ and
$$
\omega_{10}=\frac{9+27\vartheta+43\vartheta^2+20\vartheta^3+
37\vartheta^4+5\vartheta^5+32\vartheta^6+3\vartheta^7+
26\vartheta^8+\vartheta^9}{47}.
$$
The coordinates of a fundamental system
$\varepsilon_1,\dots,\varepsilon_9$ of $S$-units with respect to
this integral basis is given by
$$
[21,107,192,-5,-120,-40,84,20,30,-60],
$$
$$
[16,99,139,-56,-113,-7,56,9,14,-30],
$$
$$
[10,4,65,197,85,-110,56,34,50,-90],
$$
$$
[21,35,196,346,94,-206,129,66,97,-177],
$$
$$
[0,-53,-31,200,145,-90,14,24,35,-60],
$$
$$
[8,24,40,33,-1,-27,25,10,15,-28],
$$
$$
[15,13,118,248,78,-143,84,45,66,-120],
$$
$$
[0,1,0,0,0,0,0,0,0,0],
$$
$$
[4,19,42,0,-26,-8,17,4,6,-12],
$$
respectively (see \cite{GGy}). Put
$$
F_0=(\varepsilon_1,\varepsilon_2,\varepsilon_3,\varepsilon_4,
\varepsilon_5,\varepsilon_6,\varepsilon_7,\varepsilon_8,
\varepsilon_9).
$$
We have $N_{old}(F_0)=2.285921\dots$ and $N(F_0)=1.564168\dots$.
Then by the above algorithm starting with $F_0$, in three steps we
get a new system of fundamental $S$-units $F_3$ given by the
transformation
$$
\left(
\begin{smallmatrix}
1 & 0 & 0 & 0 & 0 & 0 &\hfill  0 & 0 &\hfill  0\\
0 & 1 & 0 & 0 & 0 & 0 &\hfill  0 & 0 &\hfill  0\\
0 & 0 & 1 & 0 & 0 & 0 &\hfill  0 & 0 &\hfill  0\\
0 & 0 & 0 & 1 & 0 & 1 &\hfill -1 & 0 &\hfill -1\\
0 & 0 & 0 & 0 & 1 & 0 &\hfill  0 & 0 &\hfill  0\\
0 & 0 & 0 & 0 & 0 & 1 &\hfill  0 & 0 &\hfill  0\\
0 & 0 & 0 & 0 & 0 & 0 &\hfill  1 & 0 &\hfill  0\\
0 & 0 & 0 & 0 & 0 & 0 &\hfill  0 & 1 &\hfill  0\\
0 & 0 & 0 & 0 & 0 & 0 &\hfill  0 & 0 &\hfill  1
\end{smallmatrix}
\right)
\cdot
\left(
\begin{smallmatrix}
1 & 0 & 0 & 0 & 1 & 0 & 0 & 0 & 0\\
0 & 1 & 0 & 0 & 0 & 0 & 0 & 0 & 0\\
0 & 0 & 1 & 0 & 0 & 0 & 0 & 0 & 0\\
0 & 0 & 0 & 1 & 0 & 0 & 0 & 0 & 0\\
0 & 0 & 0 & 0 & 1 & 0 & 0 & 0 & 0\\
0 & 0 & 0 & 0 & 0 & 1 & 0 & 0 & 0\\
0 & 0 & 0 & 0 & 0 & 0 & 1 & 0 & 0\\
0 & 0 & 0 & 0 & 0 & 0 & 0 & 1 & 0\\
0 & 0 & 0 & 0 & 0 & 0 & 0 & 0 & 1
\end{smallmatrix}
\right)
\cdot
\left(
\begin{smallmatrix}
1 & 0 & 0 & 0 & 0 & 0 & 0 & 0 & 0\\
0 & 1 & 0 & 0 & 0 & 0 & 0 & 0 & 0\\
0 & 0 & 1 & 0 & 0 & 0 & 0 & 0 & 1\\
0 & 0 & 0 & 1 & 0 & 0 & 0 & 0 & 0\\
0 & 0 & 0 & 0 & 1 & 0 & 0 & 0 & 0\\
0 & 0 & 0 & 0 & 0 & 1 & 0 & 0 & 0\\
0 & 0 & 0 & 0 & 0 & 0 & 1 & 0 & 0\\
0 & 0 & 0 & 0 & 0 & 0 & 0 & 1 & 0\\
0 & 0 & 0 & 0 & 0 & 0 & 0 & 0 & 1
\end{smallmatrix}
\right)
=
$$
$$
=
\left(
\begin{smallmatrix}
1 & 0 & 0 & 0 & 1 & 0 &\hfill  0 & 0 &\hfill  0\\
0 & 1 & 0 & 0 & 0 & 0 &\hfill  0 & 0 &\hfill  0\\
0 & 0 & 1 & 0 & 0 & 0 &\hfill  0 & 0 &\hfill  1\\
0 & 0 & 0 & 1 & 0 & 1 &\hfill -1 & 0 &\hfill -1\\
0 & 0 & 0 & 0 & 1 & 0 &\hfill  0 & 0 &\hfill  0\\
0 & 0 & 0 & 0 & 0 & 1 &\hfill  0 & 0 &\hfill  0\\
0 & 0 & 0 & 0 & 0 & 0 &\hfill  1 & 0 &\hfill  0\\
0 & 0 & 0 & 0 & 0 & 0 &\hfill  0 & 1 &\hfill  0\\
0 & 0 & 0 & 0 & 0 & 0 &\hfill  0 & 0 &\hfill  1
\end{smallmatrix}
\right)
.
$$
Actually, we have
$$
F_3=(\varepsilon_1,\varepsilon_2,\varepsilon_3,\varepsilon_4,
\varepsilon_1\varepsilon_5,\varepsilon_4\varepsilon_6,
\varepsilon_4^{-1}\varepsilon_6,\varepsilon_7,
\varepsilon_3\varepsilon_4^{-1}\varepsilon_8).
$$
By a simple calculation we get that $N(F_3)=1.209236\dots$. Using
the second part of our algorithm we obtain that $F_3$ is an
optimal system of fundamental $S$-units. In this case we performed
the reduction starting with the initial bound $B<10000=:C_{ini}$
both with $F_0$ and with $F_3$. The results of our calculations
are summarized in Table 2.

\begin{table}[htb]\centering
\vskip.15cm
\begin{tabular}{|l|l|l|l|}
\cline{1-4} \vbox to1,88ex{\vspace{1pt}\vfil\hbox to12,80ex{\hfil
\hfil}} &\vbox to1,88ex{\vspace{1pt}\vfil\hbox to20,80ex{\hfil
using $F_0$ and $N_{old}(F_0)$ \hfil}} & \vbox
to1,88ex{\vspace{1pt}\vfil\hbox to18,80ex{\hfil using $F_0$ and
$N(F_0)$ \hfil}} & \vbox to1,88ex{\vspace{1pt}\vfil\hbox
to18,80ex{\hfil using $F_3$ and $N(F_3)$ \hfil}} \\

\cline{1-4} \vbox to1,88ex{\vspace{1pt}\vfil\hbox to12,80ex{\hfil
$C^*$\hfil}} &\vbox to1,88ex{\vspace{1pt}\vfil\hbox
to20,80ex{\hfil 2.285921...\hfil}} & \vbox
to1,88ex{\vspace{1pt}\vfil\hbox to18,80ex{\hfil 1.564168...\hfil}}
& \vbox to1,88ex{\vspace{1pt}\vfil\hbox to18,80ex{\hfil
1.209236...\hfil}} \\

\cline{1-4} \vbox to1,88ex{\vspace{1pt}\vfil\hbox to12,80ex{\hfil
$C_{red}$\hfil}} &\vbox to1,88ex{\vspace{1pt}\vfil\hbox
to20,80ex{\hfil 2079\hfil}} & \vbox
to1,88ex{\vspace{1pt}\vfil\hbox to18,80ex{\hfil 1416\hfil}} &
\vbox to1,88ex{\vspace{1pt}\vfil\hbox to18,80ex{\hfil 1011\hfil}}
\\

\cline{1-4} \vbox to1,88ex{\vspace{1pt}\vfil\hbox to12,80ex{\hfil
$C^*$ ratio\hfil}} &\vbox to1,88ex{\vspace{1pt}\vfil\hbox
to20,80ex{\hfil 1\hfil}} & \vbox to1,88ex{\vspace{1pt}\vfil\hbox
to18,80ex{\hfil 0.684261...\hfil}} & \vbox
to1,88ex{\vspace{1pt}\vfil\hbox to18,80ex{\hfil 0.528993...\hfil}} \\

\cline{1-4} \vbox to1,88ex{\vspace{1pt}\vfil\hbox to12,80ex{\hfil
$C_{red}$ ratio\hfil}} &\vbox to1,88ex{\vspace{1pt}\vfil\hbox
to20,80ex{\hfil 1\hfil}} & \vbox to1,88ex{\vspace{1pt}\vfil\hbox
to18,80ex{\hfil 0.681096...\hfil}} & \vbox
to1,88ex{\vspace{1pt}\vfil\hbox to18,80ex{\hfil 0.486291...\hfil}} \\

\cline{1-4} \vbox to1,88ex{\vspace{1pt}\vfil\hbox to12,80ex{\hfil
domain ratio\hfil}} &\vbox to1,88ex{\vspace{1pt}\vfil\hbox
to20,80ex{\hfil 1\hfil}} & \vbox to1,88ex{\vspace{1pt}\vfil\hbox
to18,80ex{\hfil 0.000996...\hfil}} & \vbox
to1,88ex{\vspace{1pt}\vfil\hbox to18,80ex{\hfil 0.000002...\hfil}} \\

\cline{1-4}
\end{tabular}
\vskip.2cm
\centerline{Table 2}
\end{table}

\vskip.2cm

\noindent{\bf Example 3.} This example is from \cite{W}. Let
$\K=\Q(\vartheta)$ where
$$
\vartheta^9+\vartheta^8-8\vartheta^7-7\vartheta^6+21\vartheta^5+
15\vartheta^4-20\vartheta^3-10\vartheta^2+5\vartheta+1=0.
$$
Note that $\K$ is the maximal real subfield of the cyclotomic
field $\Q(\zeta_{19})$. Let $S$ be the set of infinite valuations
of $\K$. An integral basis of $\K$ is given by
$1,\vartheta,\dots,\vartheta^8$. The coordinates of a fundamental
system $\varepsilon_1,\dots,\varepsilon_8$ of $S$-units with
respect to this integral basis is given by
$$
[1,-4,-10,10,15,-6,-7,1,1],[0, 3, 0, -1, 0, 0, 0, 0, 0],
$$
$$
[1, -2, -3, 1, 1, 0, 0, 0, 0],[2, 0, -9, 0, 6, 0, -1, 0, 0],
$$
$$
[0, 1, 0, 0, 0, 0, 0, 0, 0],[2, 0, -1, 0, 0, 0, 0, 0, 0],
$$
$$
[2, 0, -4, 0, 1, 0, 0, 0, 0],[0, -5, 5, 10, -5, -6, 1, 1, 0],
$$
respectively; see \cite{W}. (Note that in \cite{W}
$\varepsilon_2=3\vartheta-\vartheta^4$ is written, which is a
typo.) Put
$$
F_0=(\varepsilon_1,\varepsilon_2,\varepsilon_3,\varepsilon_4,
\varepsilon_5,\varepsilon_6,\varepsilon_7,\varepsilon_8).
$$
A simple calculation shows that $N_{old}(F_0)=2.561675\dots$.
Further, we also have $N(F_0)=1.872827\dots$. Then our algorithm
in six steps yields a new system of fundamental $S$-units $F_6$
given by the transformation
$$
\left(
\begin{smallmatrix}
1 & 0 &\hfill -1 & 0 & 0 & 0 & 0 & 0\\
0 & 1 &\hfill  0 & 0 & 0 & 0 & 0 & 0\\
0 & 0 &\hfill  1 & 0 & 0 & 0 & 0 & 0\\
0 & 0 &\hfill  0 & 1 & 0 & 0 & 0 & 0\\
0 & 0 &\hfill  0 & 0 & 1 & 0 & 0 & 0\\
0 & 0 &\hfill  0 & 0 & 0 & 1 & 0 & 0\\
0 & 0 &\hfill  0 & 0 & 0 & 0 & 1 & 0\\
0 & 0 &\hfill  0 & 0 & 0 & 0 & 0 & 1
\end{smallmatrix}
\right)
\cdot
\left(
\begin{smallmatrix}
1 & 0 & 0 & 0 & 0 & 0 & 0 & 0\\
0 & 1 & 0 & 0 & 0 & 0 & 0 & 0\\
0 & 0 & 1 & 0 & 0 & 0 & 0 & 1\\
0 & 0 & 0 & 1 & 0 & 0 & 0 & 0\\
0 & 0 & 0 & 0 & 1 & 0 & 0 & 0\\
0 & 0 & 0 & 0 & 0 & 1 & 0 & 0\\
0 & 0 & 0 & 0 & 0 & 0 & 1 & 0\\
0 & 0 & 0 & 0 & 0 & 0 & 0 & 1
\end{smallmatrix}
\right)
\cdot
\left(
\begin{smallmatrix}
1 & 0 & 0 & 0 & 0 & 0 & 0 & 0\\
0 & 1 & 0 & 0 & 0 & 0 & 0 & 0\\
0 & 0 & 1 & 0 & 0 & 0 & 0 & 0\\
0 & 0 & 0 & 1 & 0 & 0 & 0 & 0\\
0 & 0 & 0 & 0 & 1 & 1 & 0 & 0\\
0 & 0 & 0 & 0 & 0 & 1 & 0 & 0\\
0 & 0 & 0 & 0 & 0 & 0 & 1 & 0\\
0 & 0 & 0 & 0 & 0 & 0 & 0 & 1
\end{smallmatrix}
\right)
\cdot
\left(
\begin{smallmatrix}
1 & 0 & 0 &\hfill  0 & 0 & 0 & 0 & 0\\
0 & 1 & 0 &\hfill  0 & 0 & 0 & 0 & 0\\
0 & 0 & 1 &\hfill  0 & 0 & 0 & 0 & 0\\
0 & 0 & 0 &\hfill  1 & 0 & 0 & 0 & 0\\
0 & 0 & 0 &\hfill  0 & 1 & 0 & 0 & 0\\
0 & 0 & 0 &\hfill  0 & 0 & 1 & 0 & 0\\
0 & 0 & 0 &\hfill  0 & 0 & 0 & 1 & 0\\
0 & 0 & 0 &\hfill -1 & 0 & 0 & 0 & 1
\end{smallmatrix}
\right)
\cdot
$$
$$
\cdot
\left(
\begin{smallmatrix}
1 & 0 & 0 & 0 & 0 & 0 & 0 & 0\\
0 & 1 & 0 & 0 & 0 & 0 & 0 & 0\\
0 & 0 & 1 & 0 & 0 & 0 & 0 & 0\\
0 & 0 & 0 & 1 & 0 & 0 & 0 & 0\\
0 & 0 & 0 & 0 & 1 & 0 & 0 & 0\\
0 & 0 & 0 & 0 & 0 & 1 & 1 & 0\\
0 & 0 & 0 & 0 & 0 & 0 & 1 & 0\\
0 & 0 & 0 & 0 & 0 & 0 & 0 & 1
\end{smallmatrix}
\right)
\cdot
\left(
\begin{smallmatrix}
\hfill 1 & 0 & 0 & 0 & 0 & 0 & 0 & 0\\
\hfill 0 & 1 & 0 & 0 & 0 & 0 & 0 & 0\\
\hfill 0 & 0 & 1 & 0 & 0 & 0 & 0 & 0\\
\hfill-1 & 0 & 1 & 1 & 1 & 1 & 1 & 1\\
\hfill 0 & 0 & 0 & 0 & 1 & 0 & 0 & 0\\
\hfill 0 & 0 & 0 & 0 & 0 & 1 & 0 & 0\\
\hfill 0 & 0 & 0 & 0 & 0 & 0 & 1 & 0\\
\hfill 0 & 0 & 0 & 0 & 0 & 0 & 0 & 1
\end{smallmatrix}
\right)
=
\left(
\begin{smallmatrix}
\hfill 0 & 0 &\hfill  0 &\hfill  1 &\hfill  1 &\hfill  1 &\hfill
1 & 0\\
\hfill 0 & 1 &\hfill  0 &\hfill  0 &\hfill  0 &\hfill  0 &\hfill
0 & 0\\
\hfill 1 & 0 &\hfill  0 &\hfill -1 &\hfill -1 &\hfill -1 &\hfill
-1 & 0\\
\hfill-1 & 0 &\hfill  1 &\hfill  1 &\hfill  1 &\hfill  1 &\hfill
1 & 1\\
\hfill 0 & 0 &\hfill  0 &\hfill  0 &\hfill  1 &\hfill  1 &\hfill
1 & 0\\
\hfill 0 & 0 &\hfill  0 &\hfill  0 &\hfill  0 &\hfill  1 &\hfill
1 & 0\\
\hfill 0 & 0 &\hfill  0 &\hfill  0 &\hfill  0 &\hfill  0 &\hfill 1
& 0\\ \hfill 1 & 0 &\hfill -1 &\hfill -1 &\hfill -1 &\hfill -1
&\hfill -1 & 0
\end{smallmatrix}
\right)
.
$$
That is, we have
$$
F_6=(\varepsilon_3\varepsilon_4^{-1}\varepsilon_8,
\varepsilon_2,\varepsilon_4\varepsilon_8^{-1},
\varepsilon_1\varepsilon_3^{-1}\varepsilon_4\varepsilon_8^{-1},
\varepsilon_1\varepsilon_3^{-1}\varepsilon_4\varepsilon_5
\varepsilon_8^{-1},
$$
$$
\varepsilon_1\varepsilon_3^{-1}\varepsilon_4\varepsilon_5
\varepsilon_6\varepsilon_8^{-1},\varepsilon_1\varepsilon_3^{-1}
\varepsilon_4\varepsilon_5\varepsilon_6\varepsilon_7
\varepsilon_8^{-1},\varepsilon_4)
$$
with $N(F_6)=1.343979\dots$. The second part of our algorithm
verifies that $F_6$ is optimal. We started the reduction with the
initial bound $B\leq 2076$ for $F_0$ (see \cite{W}), while using
the inverse of the basis reduction matrix we could start with
$B\leq 4152$ in case of $F_6$. The data and the information
derived from them are given in Table 3.

\begin{table}[htb]\centering
\vskip.15cm
\begin{tabular}{|l|l|l|l|}
\cline{1-4} \vbox to1,88ex{\vspace{1pt}\vfil\hbox to12,80ex{\hfil
\hfil}} &\vbox to1,88ex{\vspace{1pt}\vfil\hbox to20,80ex{\hfil
using $F_0$ and $N_{old}(F_0)$ \hfil}} & \vbox
to1,88ex{\vspace{1pt}\vfil\hbox to18,80ex{\hfil using $F_0$ and
$N(F_0)$ \hfil}} & \vbox to1,88ex{\vspace{1pt}\vfil\hbox
to18,80ex{\hfil using $F_6$ and $N(F_6)$ \hfil}} \\

\cline{1-4} \vbox to1,88ex{\vspace{1pt}\vfil\hbox to12,80ex{\hfil
$C^*$\hfil}} &\vbox to1,88ex{\vspace{1pt}\vfil\hbox
to20,80ex{\hfil 2.561675...\hfil}} & \vbox
to1,88ex{\vspace{1pt}\vfil\hbox to18,80ex{\hfil 1.872827...\hfil}}
& \vbox to1,88ex{\vspace{1pt}\vfil\hbox to18,80ex{\hfil
1.343979...\hfil}} \\

\cline{1-4} \vbox to1,88ex{\vspace{1pt}\vfil\hbox to12,80ex{\hfil
$C_{red}$\hfil}} &\vbox to1,88ex{\vspace{1pt}\vfil\hbox
to20,80ex{\hfil 1664\hfil}} & \vbox
to1,88ex{\vspace{1pt}\vfil\hbox to18,80ex{\hfil 1210\hfil}} &
\vbox to1,88ex{\vspace{1pt}\vfil\hbox to18,80ex{\hfil 824\hfil}}
\\

\cline{1-4} \vbox to1,88ex{\vspace{1pt}\vfil\hbox to12,80ex{\hfil
$C^*$ ratio\hfil}} &\vbox to1,88ex{\vspace{1pt}\vfil\hbox
to20,80ex{\hfil 1\hfil}} & \vbox to1,88ex{\vspace{1pt}\vfil\hbox
to18,80ex{\hfil 0.731094...\hfil}} & \vbox
to1,88ex{\vspace{1pt}\vfil\hbox to18,80ex{\hfil 0.524648...\hfil}} \\

\cline{1-4} \vbox to1,88ex{\vspace{1pt}\vfil\hbox to12,80ex{\hfil
$C_{red}$ ratio\hfil}} &\vbox to1,88ex{\vspace{1pt}\vfil\hbox
to20,80ex{\hfil 1\hfil}} & \vbox to1,88ex{\vspace{1pt}\vfil\hbox
to18,80ex{\hfil 0.727163...\hfil}} & \vbox
to1,88ex{\vspace{1pt}\vfil\hbox to18,80ex{\hfil 0.495192...\hfil}} \\

\cline{1-4} \vbox to1,88ex{\vspace{1pt}\vfil\hbox to12,80ex{\hfil
domain ratio\hfil}} &\vbox to1,88ex{\vspace{1pt}\vfil\hbox
to20,80ex{\hfil 1\hfil}} & \vbox to1,88ex{\vspace{1pt}\vfil\hbox
to18,80ex{\hfil 0.006122...\hfil}} & \vbox
to1,88ex{\vspace{1pt}\vfil\hbox to18,80ex{\hfil 0.000013...\hfil}} \\

\cline{1-4}
\end{tabular}
\vskip.2cm
\centerline{Table 3}
\end{table}

\vskip.2cm

\noindent{\bf Example 4.} This example is new. Let $\K=\Q(\vartheta)$
where
$$
\vartheta^{18}+\vartheta^{17}+\vartheta^{16}+\vartheta^{15}+
\vartheta^{14}+\vartheta^{13}+\vartheta^{12}+\vartheta^{11}+
\vartheta^{10}+\vartheta^9+
$$
$$
+\vartheta^8+\vartheta^7+\vartheta^6+\vartheta^5+\vartheta^4+
\vartheta^3+\vartheta^2+\vartheta+1=0,
$$
so $\K$ is the cyclotomic field $\Q(\zeta_{19})$ (with
$\zeta_{19}=\vartheta$). Let $S$ be the set of infinite valuations
of $\K$. An integral basis of $\K$ is given by
$1,\vartheta,\dots,\vartheta^{17}$. The coordinates of a
fundamental system $\varepsilon_1,\dots,\varepsilon_8$ of
$S$-units with respect to this integral basis (obtained by Magma
\cite{BCP}) is given by
$$
[1, 1, 1, 1, 1, 1, 1, 1, 1, 0, 1, 1, 1, 1, 1, 1, 1, 1],
$$
$$
[0, 0, 0, 1, 0, 0, 0, 0, 0, 0, 1, 0, 0, 0, 0, 0, 0, 0],
$$
$$
[1, 1, 1, 1, 1, 1, 1, 1, 1, 1, 1, 1, 1, 1, 1, 0, 1, 1],
$$
$$
[1, 1, 1, 0, 1, 1, 1, 1, 1, 1, 1, 1, 1, 1, 1, 1, 1, 1],
$$
$$
[1, 0, 1, 1, 1, 1, 1, 1, 1, 1, 1, 1, 1, 1, 1, 1, 1, 1],
$$
$$
[1, 1, 1, 1, 1, 1, 1, 1, 1, 1, 1, 1, 1, 0, 1, 1, 1, 1],
$$
$$
[1, 1, 1, 1, 1, 1, 1, 1, 1, 1, 0, 1, 1, 1, 1, 1, 1, 1],
$$
$$
[0, 0, 0, 0, 0, 0, 0, 1, 0, 0, 0, 0, 0, 1, 0, 0, 0, 0],
$$
respectively. Put
$$
F_0=(\varepsilon_1,\varepsilon_2,\varepsilon_3,\varepsilon_4,
\varepsilon_5,\varepsilon_6,\varepsilon_7,\varepsilon_8).
$$
A simple calculation shows that $N_{old}(F_0)=1.280834\dots$ and
$N(F_0)=0.936410\dots$. Then the above algorithm in seven steps
provides a new system of fundamental $S$-units $F_7$ given by the
transformation
$$
\left(
\begin{smallmatrix}
1 & 0 & 0 & 0 & 0 & 0 & 0 & 0\\
0 & 1 & 0 & 0 & 0 & 0 & 0 & 0\\
0 & 0 & 1 & 0 & 0 & 0 & 0 & 1\\
0 & 0 & 0 & 1 & 0 & 0 & 0 & 0\\
0 & 0 & 0 & 0 & 1 & 0 & 0 & 0\\
0 & 0 & 0 & 0 & 0 & 1 & 0 & 0\\
0 & 0 & 0 & 0 & 0 & 0 & 1 & 0\\
0 & 0 & 0 & 0 & 0 & 0 & 0 & 1
\end{smallmatrix}
\right)
\cdot
\left(
\begin{smallmatrix}
1 & 0 &\hfill  0 & 0 & 0 & 0 & 0 & 0\\
0 & 1 &\hfill  0 & 0 & 0 & 0 & 0 & 0\\
0 & 0 &\hfill  1 & 0 & 0 & 0 & 0 & 0\\
0 & 0 &\hfill  0 & 1 & 0 & 0 & 0 & 0\\
0 & 0 &\hfill  0 & 0 & 1 & 0 & 0 & 0\\
0 & 0 &\hfill  0 & 0 & 0 & 1 & 0 & 0\\
0 & 0 &\hfill  0 & 0 & 0 & 0 & 1 & 0\\
0 & 0 &\hfill -1 & 0 & 0 & 0 & 1 & 1
\end{smallmatrix}
\right)
\cdot
\left(
\begin{smallmatrix}
1 & 0 & 0 & 0 & 0 & 0 & 0 & 0\\
0 & 1 & 0 & 0 & 0 & 0 & 0 & 0\\
0 & 0 & 1 & 0 & 0 & 0 & 0 & 0\\
0 & 0 & 0 & 1 & 0 & 0 & 0 & 0\\
0 & 0 & 0 & 0 & 1 & 0 & 0 & 0\\
0 & 0 & 0 & 0 & 0 & 1 & 0 & 0\\
0 & 0 & 0 & 1 & 0 & 0 & 1 & 0\\
0 & 0 & 0 & 0 & 0 & 0 & 0 & 1
\end{smallmatrix}
\right)
\cdot
\left(
\begin{smallmatrix}
1 & 0 &\hfill  0 & 0 & 0 & 0 & 0 & 0\\
0 & 1 &\hfill -1 & 1 & 0 & 0 & 1 & 1\\
0 & 0 &\hfill  1 & 0 & 0 & 0 & 0 & 0\\
0 & 0 &\hfill  0 & 1 & 0 & 0 & 0 & 0\\
0 & 0 &\hfill  0 & 0 & 1 & 0 & 0 & 0\\
0 & 0 &\hfill  0 & 0 & 0 & 1 & 0 & 0\\
0 & 0 &\hfill  0 & 0 & 0 & 0 & 1 & 0\\
0 & 0 &\hfill  0 & 0 & 0 & 0 & 0 & 1
\end{smallmatrix}
\right)
\cdot
$$
$$
\cdot
\left(
\begin{smallmatrix}
1 & 0 & 0 & 0 & 0 & 0 & 0 & 0\\
0 & 1 & 0 & 0 & 0 & 0 & 0 & 0\\
0 & 0 & 1 & 0 & 0 & 0 & 0 & 0\\
0 & 0 & 0 & 1 & 1 & 0 & 0 & 0\\
0 & 0 & 0 & 0 & 1 & 0 & 0 & 0\\
0 & 0 & 0 & 0 & 0 & 1 & 0 & 0\\
0 & 0 & 0 & 0 & 0 & 0 & 1 & 0\\
0 & 0 & 0 & 0 & 0 & 0 & 0 & 1
\end{smallmatrix}
\right)
\cdot
\left(
\begin{smallmatrix}
1 & 0 & 0 & 0 & 0 & 0 & 0 & 0\\
0 & 1 & 0 & 0 & 0 & 0 & 0 & 0\\
0 & 0 & 1 & 0 & 0 & 0 & 0 & 0\\
0 & 0 & 0 & 1 & 0 & 0 & 0 & 0\\
0 & 0 & 0 & 0 & 1 & 0 & 0 & 0\\
1 & 0 & 0 & 0 & 0 & 1 & 0 & 0\\
0 & 0 & 0 & 0 & 0 & 0 & 1 & 0\\
0 & 0 & 0 & 0 & 0 & 0 & 0 & 1
\end{smallmatrix}
\right)
\cdot
\left(
\begin{smallmatrix}
1 &\hfill  0 & 0 &\hfill  0 & 0 & 0 &\hfill  0 &\hfill  0\\
0 &\hfill  1 & 0 &\hfill  0 & 0 & 0 &\hfill  0 &\hfill  0\\
0 &\hfill  0 & 1 &\hfill  0 & 0 & 0 &\hfill  0 &\hfill  0\\
0 &\hfill  0 & 0 &\hfill  1 & 0 & 0 &\hfill  0 &\hfill  0\\
1 &\hfill -1 & 1 &\hfill -1 & 1 & 1 &\hfill -1 &\hfill -1\\
0 &\hfill  0 & 0 &\hfill  0 & 0 & 1 &\hfill  0 &\hfill  0\\
0 &\hfill  0 & 0 &\hfill  0 & 0 & 0 &\hfill  1 &\hfill  0\\
0 &\hfill  0 & 0 &\hfill  0 & 0 & 0 &\hfill  0 &\hfill  1
\end{smallmatrix}
\right)
=
$$
$$
=
\left(
\begin{smallmatrix}
1 &\hfill  0 & 0 &\hfill  0 & 0 & 0 &\hfill  0 &\hfill  0\\
1 &\hfill  0 & 0 &\hfill  0 & 1 & 1 &\hfill  0 &\hfill  0\\
1 &\hfill -1 & 1 &\hfill  0 & 1 & 1 &\hfill  0 &\hfill  0\\
1 &\hfill -1 & 1 &\hfill  0 & 1 & 1 &\hfill -1 &\hfill -1\\
1 &\hfill -1 & 1 &\hfill -1 & 1 & 1 &\hfill -1 &\hfill -1\\
1 &\hfill  0 & 0 &\hfill  0 & 0 & 1 &\hfill  0 &\hfill  0\\
1 &\hfill -1 & 1 &\hfill  0 & 1 & 1 &\hfill  0 &\hfill -1\\
1 &\hfill -1 & 0 &\hfill  0 & 1 & 1 &\hfill  0 &\hfill  0
\end{smallmatrix}
\right)
.
$$
Actually, we have
$$
F_7=(\varepsilon_1\varepsilon_2\varepsilon_3\varepsilon_4
\varepsilon_5\varepsilon_6\varepsilon_7\varepsilon_8,
\varepsilon_3^{-1}\varepsilon_4^{-1}\varepsilon_5^{-1}
\varepsilon_7^{-1}\varepsilon_8^{-1},\varepsilon_3
\varepsilon_4\varepsilon_5\varepsilon_7,\varepsilon_5^{-1},
$$
$$
\varepsilon_2\varepsilon_3\varepsilon_4\varepsilon_5
\varepsilon_7\varepsilon_8,\varepsilon_2\varepsilon_3
\varepsilon_4\varepsilon_5\varepsilon_6\varepsilon_7
\varepsilon_8,\varepsilon_4^{-1}\varepsilon_5^{-1},
\varepsilon_4^{-1}\varepsilon_5^{-1}\varepsilon_7^{-1})
$$
with $N(F_7)=0.67198843\dots$. By the second part of our algorithm
we get that $F_7$ is an optimal system of fundamental $S$-units.
To execute the reduction, we used the initial bound $B<10000$ both
with $F_0$ and with $F_7$. The outcome of our calculations is
contained in Table 4.

\begin{table}[htb]\centering
\vskip.15cm
\begin{tabular}{|l|l|l|l|}
\cline{1-4} \vbox to1,88ex{\vspace{1pt}\vfil\hbox to12,80ex{\hfil
\hfil}} &\vbox to1,88ex{\vspace{1pt}\vfil\hbox to20,80ex{\hfil
using $F_0$ and $N_{old}(F_0)$ \hfil}} & \vbox
to1,88ex{\vspace{1pt}\vfil\hbox to18,80ex{\hfil using $F_0$ and
$N(F_0)$ \hfil}} & \vbox to1,88ex{\vspace{1pt}\vfil\hbox
to18,80ex{\hfil using $F_7$ and $N(F_7)$ \hfil}} \\

\cline{1-4} \vbox to1,88ex{\vspace{1pt}\vfil\hbox to12,80ex{\hfil
$C^*$\hfil}} &\vbox to1,88ex{\vspace{1pt}\vfil\hbox
to20,80ex{\hfil 1.280834...\hfil}} & \vbox
to1,88ex{\vspace{1pt}\vfil\hbox to18,80ex{\hfil 0.936410...\hfil}}
& \vbox to1,88ex{\vspace{1pt}\vfil\hbox to18,80ex{\hfil
0.671988...\hfil}} \\

\cline{1-4} \vbox to1,88ex{\vspace{1pt}\vfil\hbox to12,80ex{\hfil
$C_{red}$\hfil}} &\vbox to1,88ex{\vspace{1pt}\vfil\hbox
to20,80ex{\hfil 792\hfil}} & \vbox to1,88ex{\vspace{1pt}\vfil\hbox
to18,80ex{\hfil 550\hfil}} & \vbox to1,88ex{\vspace{1pt}\vfil\hbox
to18,80ex{\hfil 386\hfil}} \\

\cline{1-4} \vbox to1,88ex{\vspace{1pt}\vfil\hbox to12,80ex{\hfil
$C^*$ ratio\hfil}} &\vbox to1,88ex{\vspace{1pt}\vfil\hbox
to20,80ex{\hfil 1\hfil}} & \vbox to1,88ex{\vspace{1pt}\vfil\hbox
to18,80ex{\hfil 0.731094...\hfil}} & \vbox
to1,88ex{\vspace{1pt}\vfil\hbox to18,80ex{\hfil 0.524649...\hfil}}
\\

\cline{1-4} \vbox to1,88ex{\vspace{1pt}\vfil\hbox to12,80ex{\hfil
$C_{red}$ ratio\hfil}} &\vbox to1,88ex{\vspace{1pt}\vfil\hbox
to20,80ex{\hfil 1\hfil}} & \vbox to1,88ex{\vspace{1pt}\vfil\hbox
to18,80ex{\hfil 0.694444...\hfil}} & \vbox
to1,88ex{\vspace{1pt}\vfil\hbox to18,80ex{\hfil 0.487373...\hfil}}
\\

\cline{1-4} \vbox to1,88ex{\vspace{1pt}\vfil\hbox to12,80ex{\hfil
domain ratio\hfil}} &\vbox to1,88ex{\vspace{1pt}\vfil\hbox
to20,80ex{\hfil 1\hfil}} & \vbox to1,88ex{\vspace{1pt}\vfil\hbox
to18,80ex{\hfil 0.002938...\hfil}} & \vbox
to1,88ex{\vspace{1pt}\vfil\hbox to18,80ex{\hfil 0.000010...\hfil}} \\

\cline{1-4}
\end{tabular}
\vskip.2cm
\centerline{Table 4}
\end{table}

\section{Proofs}

In this section we give the proofs of our results.

\begin{proof}[Proof of Proposition \ref{p1}.]
Noting that the rank of $R_F$ is $s-1$ and that by the product
formula \eqref{pf} we have $\sum\limits_{j=1}^s
\log|\varepsilon_i|_{v_j}=0$ for each $i\in\{1,\dots,s-1\}$, the
statement follows from the elementary theory of systems of linear
equations.
\end{proof}

\begin{proof}[Proof of Proposition \ref{p2}.]
In view of Proposition \ref{p1}, the only freedom we have in the
choice of $R_F'$ is to choose $u_1,\dots,u_{s-1}$. Take any
$i\in\{1,\dots,s-1\}$, and let $z_1,\dots,z_s$ be the
rearrangement of the entries of $\w_i$ such that
$z_1\leq\dots\leq z_s$. Writing $w_{i,s}=0$, it is obvious that
$$
\min\limits_{u_i\in\R}\sum\limits_{j=1}^s |w_{i,j}-u_i|
$$
is achieved by the choice $u_i=z_l$, where
$l=\lfloor(s+1)/2\rfloor$. Hence the statement follows.
\end{proof}

\begin{proof}[Proof of Theorem \ref{t1}.]
Let $T=(\eta_1,\dots,\eta_{s-1})$ be an arbitrary system of
fundamental $S$-units, and suppose that for the system $F$ we have
$N(F)\leq c$. Then we have $R_F=R_T A$ with some integral
unimodular matrix $A$ of size $(s-1)\times (s-1)$. Obviously, it
is sufficient to "bound" $A$. Observe that in view of Proposition
\ref{p2}, $N(T)$ and $N(F)$ are exclusively ruled by
$$
W_T:=
\begin{pmatrix}
\w_1^{(T)}\\
\vdots\\
\w_{s-1}^{(T)}
\end{pmatrix}
\ \ \
\mbox{and}
\ \ \
W_F:=
\begin{pmatrix}
\w_1^{(F)}\\
\vdots\\
\w_{s-1}^{(F)}
\end{pmatrix},
$$
respectively (with the obvious notation). Since $W_F=A^{-1}W_T$,
$W_T$ and $W_F$ are bases of the same lattice. As the last entry
of each $\w_i^{(F)}$ is zero, the restriction $|\w_i^{(F)}|_C\leq
c$ yields that the absolute values of the components of
$\w_i^{(F)}$ are all at most $c$. That is, $W_F$ is a basis of a
fixed lattice, within a bounded region, and our statement follows.
\end{proof}

\begin{proof}[Proof of Lemma \ref{l1}.]
The statement is a trivial consequence of the definition
\eqref{cnorm} and of the fact that (using the notation above
\eqref{cnorm}) by symmetry
$$
\vert\x\vert_C=\sum\limits_{i=1}^n |y_{l'}-y_i|
$$
holds, where $l'=\lfloor(n+2)/2\rfloor$. \qed
\end{proof}

\begin{proof}[Proof of Lemma \ref{l2}.]
First we prove that any convex linear combination of the vectors
in $T$ belongs to $H$. For this purpose, choose arbitrary
non-negative real numbers
$\lambda_0^+,\lambda_0^-,\lambda_1^+,\lambda_1^-,\dots,
\lambda_{n-1}^+,\lambda_{n-1}^-$ such that
$\sum\limits_{i=0}^{n-1} (\lambda_i^+ + \lambda_i^-) = 1$. Put
$$
\x=\sum\limits_{i=0}^{n-1} (\lambda_i^+ \ee_i + \lambda_i^-
(-\ee_i)).
$$
Then for each $i\in\{1,\dots,n-1\}$ the $i$-th entry $x_i$ of $\x$
is given by $x_i=\lambda_0^+ - \lambda_0^- + \lambda_i^+ -
\lambda_i^-$, while $x_n=0$. Hence clearly,
$$
\vert\x\vert_C= \vert(\lambda_1^+ -
\lambda_1^-,\dots,\lambda_{n-1}^+ - \lambda_{n-1}^-,\lambda_0^- -
\lambda_0^+)\vert_C.
$$
Thus
$$
\vert\x\vert_C\leq \sum\limits_{i=0}^{n-1}|\lambda_i^+
-\lambda_i^-|\leq \sum\limits_{i=0}^{n-1}(\lambda_i^+
+\lambda_i^-)=1,
$$
and the statement follows.

Let now $\x\in H$ be arbitrary, and write $\x=(x_1,\dots,x_n)$.
Let $y_1,\dots,y_n$ be a rearrangement of $x_1,\dots,x_n$ such
that $y_1\leq\dots\leq y_n$ and write $x_i=y_{\nu(i)}$, where
$\nu$ is the underlying permutation of the indices $1,\dots,n$.
(Note that $y_{\nu(n)}=x_n=0$.) As usual, put
$l=\lfloor(n+1)/2\rfloor$, for $i=1,\dots,n-1$ write
$\lambda_i^+=\max\{y_{\nu(i)}-y_l,0\}$ and
$\lambda_i^-=\max\{y_l-y_{\nu(i)},0\}$, and set
$\lambda_0^+=\max\{y_l,0\}$ and $\lambda_0^-=\max\{-y_l,0\}$. Then
on the one hand, $x_i=(y_{\nu(i)}-y_l)+y_l$ implies
$$
\x=\sum\limits_{i=0}^{n-1} (\lambda_i^+\ee_i + \lambda_i^-
(-\ee_i)).
$$
Further, on the other hand
$$
\sum\limits_{i=0}^{n-1} (\lambda_i^+ + \lambda_i^-)
=\sum\limits_{i=1}^n |y_l-y_{\nu(i)}|=\vert\x\vert_C\leq 1
$$
as $\x\in H$, and the lemma follows.
\end{proof}

\begin{proof}[Proof of Lemma \ref{l3}.]
Let ${\mathbf a}_i$ be the $i$-th row of $A$ $(i=1,\dots,s-1)$.
Using Proposition \ref{p1} without loss of generality we may
assume that all entries of the last column of $R_F'$ are zero.
Hence by $||(A R_F')^*||<N(F)$ Lemmas \ref{l1} and \ref{l2} imply
that ${\mathbf a}_i R_F'$ belongs to the convex hull of the set
$$
\{\pm N(F)\ee_0,\pm N(F)\ee_1,\dots,\pm N(F)\ee_{s-1}\}
$$
(where the $\ee_i$ are defined in Lemma \ref{l2}). However, then
using $R_F'\cdot R_F=E_{(s-1)\times (s-1)}$ we get that ${\mathbf
a}_i$ belongs to the convex hull of the set
$$
\{\pm N(F){\mathbf b}_0,\pm N(F){\mathbf b}_1,\dots, \pm
N(F){\mathbf b}_{s-1}\},
$$
where ${\mathbf b}_0={\mathbf b}_1+\dots+{\mathbf b}_{s-1}$. Since
by the product formula \eqref{pf} we have ${\mathbf b}_s=-{\mathbf
b}_0$, the lemma follows.
\end{proof}

\section{Effects on the method of Wildanger and Smart}
\label{wildsect}

In this section we discuss about the effects of the results in
the paper on the method of Wildanger \cite{W} and Smart \cite{Sm2}
for the resolution of $S$-unit equations. For this purpose, first
we briefly and schematically outline the main steps of the method.
After that we show that our results yield certain improvements of
the method. At this point the author would like to express his deep
thanks to Attila Peth\H o for the many fruitful and stimulating
discussions and advice about the contents of this section.

\subsection{The method of Wildanger and Smart}

Now we briefly sketch the method worked out by Wildanger \cite{W}
and Smart \cite{Sm2} for the resolution of $S$-unit equations. We
follow the presentation in \cite{Sm2}, with certain simplifications.
First we need to introduce some notation.

Let $\K$ and $S$ be as before. For $K\in \R$ with $K>1$ write
$$
\langle\langle K,S \rangle\rangle= \{\alpha\in\K:1/K\leq
|\alpha|_v\leq K\ \text{for all}\ v\in S\}.
$$
Let ${\mathcal L}$ denote the set of solutions of the $S$-unit
equation \eqref{sueq}, that is
$$
{\mathcal L}=\{(x_1,x_2)\in U_S\times U_S:\alpha_1 x_1 + \alpha_2
x_2 = 1\}.
$$
Further, put
$$
{\mathcal L}_{H_i}=\{(x_1,x_2)\in {\mathcal L}:B\leq H_i\}
$$
where $B$ is defined after \eqref{xiexpand}, and set
$$
{\mathcal L}_{H_i}(K)=\{(x_1,x_2)\in {\mathcal
L}_{H_i}:x_1\in\langle\langle K,S\rangle\rangle\}.
$$
Starting from \eqref{sueq} after making the standard steps (see
section 1.2) we arrive at \eqref{matrixeq}. Write ${\mathcal
E}_1,\dots,{\mathcal E}_{s-1}$ for the columns of the matrix at
the left hand side of \eqref{matrixeq}, and ${\mathcal X}$ for the
vector on the right hand side. Then \eqref{matrixeq} can be
rewritten as
\begin{equation}
\label{vecteq}
b_{1,1}{\mathcal E}_1+\dots+b_{1,s-1}{\mathcal E}_{s-1}=
{\mathcal X}.
\end{equation}

As earlier, denote by $C_{red}$ the reduced bound obtained for
$B$ (defined after \eqref{xiexpand}) after executing Baker's method
and the LLL-algorithm. Put
\begin{equation}
\label{w1} K_0 = \max\limits_{v\in S} \exp(C_{red}|\log
|\varepsilon_1|_v|+\dots+ C_{red}|\log |\varepsilon_{s-1}|_v|).
\end{equation}
Then by the help of \eqref{vecteq} one can easily get that
\begin{equation}
\label{k0ineq} {\mathcal L}={\mathcal L}_{H_0}(K_0)
\end{equation}
with $H_0=C_{red}$ (see Lemma 1 of \cite{Sm2}).

Put
$$
s_i=\max\limits_{v\in S}\max(|\alpha_i|_v,|\alpha_i^{-1}|_v)\
\text{for}\ i=1,2
$$
and
$$
s_3=\max\limits_{v\in S}\min(|\alpha_2^{-1}|_v).
$$
Now let $K_i,K_{i+1}$ be real numbers such that
$\max(s_1,s_2,s_3,(s_3-1)/s_1)<K_{i+1}<K_i$ and let $H_i\in\Z$.
Note that we have $K_{i+1}>1$. For $v\in S$ define the sets
$$
T_{1,v}(H_i,K_i,K_{i+1})=\left\{ (x_1,x_2)\in {\mathcal
L}_{H_i}(K_i):|-\alpha_1 x_1-1|_v<{\frac{1}{1+s_1 K_{i+1}}}
\right\},
$$
$$
T_{2,v}(H_i,K_i,K_{i+1})=\left\{ (x_1,x_2)\in {\mathcal
L}_{H_i}(K_i):\left|-{\frac{1}{\alpha_1 x_1}}-1\right|_v<
{\frac{1}{1+s_1K_{i+1}}} \right\},
$$
\begin{multline*}
T_{3,v}(H_i,K_i,K_{i+1})=\{(x_1,x_2)\in {\mathcal
L}_{H_i}(K_i):|-\alpha_2 x_2-1|_v<{\frac{s_1}{K_{i+1}}},\\
\alpha_2 x_2\in \langle\langle 1+s_1K_i,S\rangle\rangle\},
\end{multline*}
\begin{multline*}
T_{4,v}(H_i,K_i,K_{i+1})=\{ (x_1,x_2)\in {\mathcal
L}_{H_i}(K_i):\left|-{\frac{\alpha_2 x_2}{\alpha_1
x_1}}-1\right|_v<{\frac{s_1}{K_{i+1}}},\\
{\frac{\alpha_2 x_2}{\alpha_1 x_1}}\in
\langle\langle 1+s_1K_i,S\rangle\rangle\}.
\end{multline*}
Further, for $j=1,2,3,4$ let
$$
T_j(H_i,K_i,K_{i+1})=\bigcup\limits_{v\in S}
T_{j,v}(H_i,K_i,K_{i+1}).
$$
Then by Lemma 2 of \cite{Sm2} we have the decomposition
\begin{equation}
\label{decomp}
{\mathcal L}_{H_i}(K_i)={\mathcal
L}_{H_{i+1}}(K_{i+1}) \bigcup\limits_{j=1}^4 T_j(H_i,K_i,K_{i+1}),
\end{equation}
with
\begin{equation}
\label{newhi}
H_{i+1}=C^*\cdot C^+
\end{equation}
where $C^*$ is the crucial constant investigated in the paper (see
\eqref{mainineq}) and
\begin{equation}
\label{cplus} C^+=\max\left(
\log\left({\frac{s_1K_{i+1}+1}{s_2}}\right),
\log\left({\frac{s_1K_{i+1}+1}{s_3}}\right),
\log(K_{i+1})\right).
\end{equation}
For more details and explanation see \cite{W} and \cite{Sm2}. Now
what happens, is that under certain conditions (which usually hold
in the first part of the algorithm; see Lemma 3 of \cite{Sm2}) the
set $\bigcup\limits_{j=1}^4 T_j(H_i,K_i,K_{i+1})$ occurring in
\eqref{decomp} proves to be empty. This is not the case in the
later part of the algorithm. However, at that stage one can rather
efficiently use the method of Fincke and Pohst \cite{FP} to
completely enumerate this set. Hence starting from $i=0$, by the
repeated application of \eqref{decomp} the whole procedure can be
iterated. Finally we are left only with a set
${\mathcal L}_{H_j}(K_j)$ for some small values of $K_j$ and $H_j$,
which can also be enumerated without any trouble. Hence in this
way we can get all solutions of the original $S$-unit equation
\eqref{sueq}.

\subsection{Some improvements}

In this subsection we indicate at which points our results may
improve upon the above described method of Wildanger and Smart.

By our method presented in the paper we are able to minimize
the value of $C_{red}$. Hence the initial value of $K_0$ in
\eqref{w1} can be taken much smaller than previously. In
particular, one may even use his/her original fundamental system
$F_0$ of units. Working with the value $N(F_0)$ defined in the
paper instead of $N_{old}(F_0)$ used earlier, we get a better
$C_{red}$ than previously. As in this case all the other
parameters in \eqref{w1} are unchanged, we already obtain some
improvement.

Furthermore, in the definition \eqref{newhi} of $H_{i+1}$ the
constant $C^*$ is used. As by our results one can take a much
smaller value for this parameter than earlier, in each iteration
of the algorithm we get a smaller value for $H_{i+1}$, and hence
the procedure can be made to "converge" faster. Here the above
remark applies again: using the original non-optimized system
$F_0$, but working with the new norm $N(F_0)$, one already gets
some improvement.

\section{Appendix}

This final section has two distinct parts. In the first subsection
we provide the reduction lemmas we used to obtain the reduced
bounds $C_{red}$ for $B$, both in the finite and in the infinite
case. In the second subsection we show how one can adjust the
method for other choices of the valuations. (We give our
motivation for doing so, as well.)

\subsection{Reduction}
\label{redu}

There are very many variants of reduction lemmas, both in the
finite and in the infinite case; we chose a lemma of Smart
\cite{Sm1} in the $p$-adic case, and a result of Ga\'al and Pohst
\cite{GP} in the complex case. Note that these lemmas have to be
applied for each possible choice of the valuation $v\in S$, and
the final reduced bound will be the maximum of the bounds obtained
for each separate $v$.

\subsubsection{The $p$-adic case}
To execute the reduction, in the $p$-adic case we use a lemma of
Smart \cite{Sm1}, which is based upon results of de Weger
\cite{deW}. For its formulation we need some preparation, in which
we follow the presentation in \cite{Sm1} with slight
modifications. (For more details see \cite{Sm1}.)

Let $P$ be a prime ideal in $\K$, corresponding to a finite
valuation $v$ in $S$. Suppose that $P$ lies above the rational
prime $p$, and suppose that $\mbox{ord}_p(\Lambda)\geq C_1B-C_2$
with some constants $C_1>0$ and $C_2$, where
$\Lambda=\alpha_1x_1=1-\alpha_2 x_2$ in \eqref{sueq}. Now assuming
that $B$ is not too small (otherwise we can derive a much better
bound for $B$ than with the reduction), we can find $\mu_i\in\K$
$(i=0,1,\dots,s'-1)$ with $s'=s$ or $s-1$, such that
$$
\alpha_2 x_2 = \mu_0 \prod\limits_{i=1}^{s'} \mu_i ^{k_i}
$$
where $k_i\in\Z$ with $|k_i|\leq B$ $(i=1,\dots,s')$. Note that
these $\mu_i$ can actually be found (see \cite{Sm1}).

Let $\Q_p$ and $\K_P$ denote the $p$-adic closure of $\Q$ and the
$P$-adic closure of $\K$, respectively. Then we can write
$\K_P=\Q_p(\phi)$ with some $\phi\in\K_p$; put $n_0:=[\K_P:\Q_p]$.
Further, set
$$
\Delta=\log_p \mu_0 + \sum\limits_{i=1}^{s'} k_i\log_p \mu_i \in \K_P.
$$
Then we can write
$$
\Delta=\sum\limits_{i=0}^{n_0-1} \Delta_i \phi^i
$$
where
$$
\Delta_i=\beta_{0,i} + \sum\limits_{j=1}^{s'} k_j\beta_{j,i}
$$
with the corresponding $\beta_{ji}\in\Q_p$ $(i=0,1,\dots,n_0-1)$.
Let $\lambda\in\Q_p$ such that
$$
\mbox{ord}_p(\lambda)=\min\limits_{1\leq j\leq s'}\left(
\min\limits_{0\leq i\leq n_0-1}(\mbox{ord}_p(\beta_{j,i}))
\right).
$$
Then (assuming again without loss of generality that $B$ is "not
too small") by a simple calculation, including Evertse's trick
(see \cite{TW} and \cite{Sm1}) we get that
$$
\mbox{ord}_p(\Delta_i/\lambda)\geq C_1B-C_3\ \ \
(i=0,1,\dots,n_0-1).
$$
Here $C_3$ is a constant which can be explicitly given in terms of
$C_2,\lambda,\phi$. Write
$$
\Delta_i/\lambda=\kappa_{0,i} + \sum\limits_{j=1}^{s'}
k_j\kappa_{j,i},\ \ \ \kappa_{j,i}\in\Z_p,\ 0\leq i\leq n_0-1
$$
with the obvious notation.

For $\gamma\in\Z_p$ and $u\in\Z$ let $\gamma^{(u)}$ denote the
unique rational integer such that $0\leq \gamma^{(u)}\leq p^u-1$
and $\gamma\equiv\gamma^{(u)} \pmod{p^u}$. Further, for $u\in\Z$
set
$$
L=
\begin{pmatrix}
1                      &        &                         &     &
& 0   \\
                       & \ddots &                         &     &
                       &
                       \\
0                      &        & 1                       &     &
&     \\
\kappa_{1,0}^{(u)}     & \dots & \kappa_{s',0}^{(u)}     & p^u &
& 0   \\
\vdots                 &        & \vdots                  &     &
\ddots &     \\ \kappa_{1,n_0-1}^{(u)} & \dots &
\kappa_{s',n_0-1}^{(u)} & 0   &        & p^u
\end{pmatrix}
\in \Z^{(s'+n_0)\times (s'+n_0)}
$$
and
$$
\underline{y}=
\begin{pmatrix}
0\\
\vdots\\
0\\
-\kappa_{0,0}^{(u)}\\
\vdots\\
-\kappa_{0,n_0-1}^{(u)}
\end{pmatrix}
\in \Z^{s'+n_0}
.
$$
Let $\mathcal{L}$ denote the lattice generated by the column
vectors of $L$ over $\Z$, and set
$$
\ell(\mathcal{L},\underline{y})=
\begin{cases}
\min\limits_{\underline{x}\in\mathcal{L},\underline{x}\neq
\underline{0}} ||\underline{x}||& \mbox{ if }
\underline{y}=\underline{0},\\
\min\limits_{\underline{x}\in\mathcal{L}} ||\underline{x}||&\mbox{
otherwise.}
\end{cases}
$$

The following result is Lemma 5 in \cite{Sm1}. Note that the
statement is in fact a consequence of Lemmas 3.4, 3.5 and 3.6 of
\cite{deW}.

\begin{lemma}
\label{redpadic} Using the previous notation, suppose that
$\mbox{ord}_p(\Lambda)\geq C_1B-C_2$ with $B\leq X_0$. Then
$\ell(\mathcal{L},\underline{y})>\sqrt{s'}X_0$ implies that
$B<(u+C_3)/C_1$.
\end{lemma}

To use this lemma, recall that by \eqref{x1vfelso} we have
$$
|\alpha_1x_1|_v \leq
|\alpha_1|_v\exp\left(\frac{-B}{(s-1)C^*}\right)
$$
for some $v\in S$. Assuming that $v$ is the valuation occurring in
the above argument, this yields
$$
p^{-f_pe_p\mbox{\scriptsize ord}_p(\alpha_1x_1)}\leq
|\alpha_1|_v\exp\left(\frac{-B}{(s-1)C^*}\right),
$$
where $\mbox{Norm}(P)=p^{f_p}$ and $e_p$ is the ramification index
of $P$. Hence a simple calculation gives
$$
\mbox{ord}_p(\Lambda)\geq \left(\frac{1}{f_p e_p\log p (s-1)
C^*}\right)B- \left(\frac{\log |\alpha_1|_v}{f_p e_p\log
p}\right).
$$
Now we can apply Lemma \ref{redpadic} with the above inequality,
using that $B\leq X_0$ holds for some constant $X_0$.

\subsubsection{The complex case} To execute the reduction in the
complex case we use the following result, which is an immediate
consequence of Lemma 1 of Ga\'al and Pohst \cite{GP}.

\begin{lemma}
\label{redabs} Let $\xi_1,\dots,\xi_k$ be non-zero real numbers,
and let $x_1,\dots,x_k$ be integers. Put
$X=\max(|x_1|,\dots,|x_k|)$. Suppose that
$$
|x_1\xi_1+\dots+x_k\xi_k|<C_2\exp(-C_1X) \mbox{ and } X<X_0
$$
hold with some positive constants $C_1,C_2$ and $X_0$. Further,
let $b_1$ be the first vector of an LLL-reduced basis of the
lattice spanned by the columns of the $k\times (k+1)$ type matrix
$$
\begin{pmatrix}
1 & 0 & \dots & 0\\
0 & 1 & \dots & 0\\
\vdots & \vdots & \ddots & \vdots\\
0 & 0 & \dots & 1\\
H\xi_1 & H\xi_2 & \dots & H\xi_k
\end{pmatrix}
$$
where $H$ is some positive constant. Then
$$
|b_1|\geq \sqrt{(k+1)2^{k-1}}\cdot X_0
$$
implies that
$$
X\leq\frac{\log H+\log C_2-\log X_0}{C_1}.
$$
\end{lemma}

Now we briefly explain how to apply this lemma. For this purpose,
let $v\in S$ be an infinite valuation such that $|1-\alpha_2
x_2|_v\leq C_2\exp(-C_1B)$ (see \eqref{bfelso}). Then using the
inequality $|\log x|\leq 2|x-1|$ which holds for $|x-1|<0.795$, we
get
$$
|b_{2,1}\log|\varepsilon_1|_v+
\dots+b_{2,s-1}\log|\varepsilon_{s-1}|_v|
\leq 2C_2\exp(-C_1B).
$$
Then we can apply Lemma \ref{redabs} with the previous inequality,
knowing that $B<X_0$ for some constant $X_0$.

\subsection{Adjusting the method for other choices of the
valuations}

In this subsection we indicate how one can adjust our results for
other choices of the valuations. The motivation is that if $S$
contains only the infinite places (i.e. we are interested in pure
unit equations) then there is an alternative, also natural choice
for the valuations: we simply take the absolute value of the real
conjugates and the absolute value of one from each pair of complex
conjugates of elements $\alpha\in\K$, without squaring in the
complex cases. To be more general, keep the previous notation, and
take arbitrary non-zero rational numbers $r_1,\dots,r_s$. Choose
now valuations $|.|_{v_i'}$ such that
$|\alpha|_{v_i'}:=|\alpha|_{v_i}^{r_i}$ for all $\alpha\in\K$,
where $|\alpha|_{v_i}$ is the previously defined ("standard")
valuation corresponding to $v_i\in S$ $(i=1,\dots,s)$. For
simplicity, we do not introduce new notation but use the previous
one, with the convention that everything is adopted for this new
choice of the valuations. We have the following variants of
Propositions \ref{p1} and \ref{p2}.

\begin{prp}
\label{p3} Let $R_F'$ be a matrix for which \eqref{rf'} is valid.
Then for each $i\in\{1,\dots,s-1\}$ there exists a $u_i\in\R$
such that the $i$-th row of $R_F'$ is of the form $\w_i-u_i\cdot
(1/r_1,\dots,1/r_s)$ with $\w_i=(w_{i,1},\dots,w_{i,s-1},0)$,
where
$$
\begin{pmatrix}
w_{1,1}&\dots&w_{1,s-1}\\
\vdots&\ddots&\vdots\\
w_{s-1,1}&\dots&w_{s-1,s-1}
\end{pmatrix}
=
\begin{pmatrix}
\log |\varepsilon_1|_{v'_1} & \dots & \log
|\varepsilon_{s-1}|_{v'_1}\\
\vdots & \ddots & \vdots \\
\log |\varepsilon_1|_{v'_{s-1}} & \dots & \log
|\varepsilon_{s-1}|_{v'_{s-1}}
\end{pmatrix}
^{-1}.
$$
\end{prp}

Note that using the assertion $\sum\limits_{j=1}^s
\frac{1}{r_j}\log|\varepsilon_i|_{v_j'}=0$ $(i=1,\dots,s-1)$, the
statement can be proved in a similar manner as Proposition
\ref{p1}. We suppress the details.

To formulate our last statement, write $r_i=q_i/p_i$ with
$p_i,q_i\in\Z$, $q_i>0$, $\gcd(p_i,q_i)=1$ and put $t_i=q/r_i$
where $q=\gcd(q_1,\dots,q_s)$. Finally, let $\w'_i$ be the
$\left(\sum\limits_{i=1}^s |t_i|\right)$-tuple such that the first
$|t_1|$ entries of $\w'_i$ equal $r_1w_{i,1}$, the next $|t_2|$
entries of $\w'_i$ equal $r_2w_{i,2}$, etc. $(i=1,\dots,s)$, with
the convention $w_{i,s}=0$.

\begin{prp}
\label{p4}
Using the above notation, we have
$$
N(F)=\frac{\max\limits_{1\leq i\leq s-1} |\w'_i|_C}{q}.
$$
\end{prp}

Observe that for $i=1,\dots,s$ we have $\sum\limits_{j=1}^{s}
|w_{i,j}-\frac{1}{r_j} u_i|= \frac{1}{q}\sum\limits_{j=1}^{s}
|t_j||r_jw_{i,j}-u_i|$. Using this assertion, Proposition \ref{p4}
can be verified similarly to Proposition \ref{p2}. We omit the
details once again.

Finally, we note that as one can easily see, using Propositions
\ref{p3} and \ref{p4} in place of Propositions \ref{p1} and
\ref{p2}, respectively, all the arguments of the paper remain
valid also under this general choice of the valuations - after
making the necessary (but rather obvious) alternations.

\section{Acknowledgements}

The author would like to express his special thanks to Attila
Peth\H o for the several discussions and advice, in particular
concerning Section \ref{wildsect} of the paper.

\end{document}